\documentclass[11pt,twoside, reqno]{amsart}

\usepackage{amssymb}
\usepackage{amsmath}
\usepackage{mathrsfs}
\usepackage{amsthm}
\usepackage{amsfonts}
\usepackage{txfonts}
\usepackage{enumerate}
\usepackage[retainorgcmds]{IEEEtrantools}

\usepackage{latexsym}
\usepackage{color}
\usepackage{graphicx}

\usepackage{pgf,tikz}
\usepackage{mathrsfs}
\usetikzlibrary{arrows}

\allowdisplaybreaks

\definecolor{ttttff}{rgb}{0.2,0.2,1.}
\definecolor{ttffcc}{rgb}{0.2,1.,0.8}
\definecolor{qqqqff}{rgb}{0.,0.,1.}

\definecolor{zzttqq}{rgb}{0.6,0.2,0.}
\definecolor{qqqqff}{rgb}{0.,0.,1.}

\newcommand\mstrut{{\phantom{.}}}
\catcode`\@=11
\def\acong{\mathrel{\mathpalette\@avereq\sim}} % isomorphic sign
\def\@avereq#1#2{\lower.5\p@\vbox{\baselineskip\z@skip\lineskip-.5\p@
    \ialign{$\m@th#1\hfil##\hfil$\crcr#2\crcr\longrightarrow\crcr}}}
\def\mequal{\mathrel{\mathpalette\@mvereq{\hbox{\tiny m}}}}
\def\@mvereq#1#2{\lower.5\p@\vbox{\baselineskip\z@skip\lineskip1.5\p@
    \ialign{$\m@th#1\hfil##\hfil$\crcr#2\crcr=\crcr}}}
\newcommand\rightinj{\lhook\joinrel\rightarrow}

\usepackage{indentfirst}

\newtheorem{theorem}{Theorem}[section]

\newtheorem{corollary}[theorem]{Corollary}
\newtheorem{proposition}[theorem]{Proposition}
\theoremstyle{definition}
\newtheorem{remark}[theorem]{Remark}

\newtheorem{example}[theorem]{Example}
\numberwithin{equation}{section}

\newcommand\E{\mathbb{E}}\newcommand\M{\mathbb{M}}
\newcommand\N{\mathbb{N}}\newcommand\R{\mathbb{R}}

\newcommand\scrF{\mathscr F}

\newcommand\sH{\mathcal H}

\def\tr{{\operatorname{tr}}}

\def\vol{{\operatorname{vol}}}
\newcommand\1{\hbox{\kern.375em\vrule height1.57ex depth-.1ex
width.05em\kern-.375em \rm 1}}

\def\Ric{{\operatorname{Ric}}} \def\Aut{{\operatorname{Aut}}}
 \def\Hess{{\operatorname{Hess}}}

\def\mathpal#1{\mathop{\mathchoice{\text{\rm #1}}%
    {\text{\rm #1}}{\text{\rm #1}}%
    {\text{\rm #1}}}\nolimits}
 \def\id{{\mathpal{id}}}

\def\bolddot{{\displaystyle\boldsymbol{.}}}

\newcommand\newdot{{\kern.8pt\cdot\kern.8pt}}
\def\partr#1#2{/\!/_{\!#1,#2}^{\phantom{.}}}
\def\partrinv#1#2{/\!/_{\!#1,#2}^{-1}}

\def\r{\right}
\def\l{\left}
\def\e{\operatorname{e}}
\def\mathpal#1{\mathop{\mathchoice{\text{\rm #1}}%
      {\text{\rm #1}}{\text{\rm #1}}%
      {\text{\rm #1}}}\nolimits}
      \def\id{{\mathpal{id}}}

\def\<{\langle}
\def\>{\rangle}

\def\1{\mathds{1}}

\makeatother

\begin{document}
 %%%%%%%%%%%%%%%%%%%%%%%%%%%%%%%%%%%%%%%%%%%%%%%%%%%%%%%%%%%%%%%%%%%%%%%

 \title[Dimension-free Harnack inequalities for conjugate heat
 equations] {Dimension-free Harnack inequalities for conjugate heat
   equations and their applications to geometric flows}

 \author[L.-J.~Cheng]{Li-Juan Cheng$^*$} \thanks{* Corresponding
   author} \address{Department of Applied Mathematics, Zhejiang
   University of Technology, Hangzhou 310023, China}
 \email{chenglj@zjut.edu.cn}
 % \curraddr{Department of Mathematics, University of Luxembourg,
 % L--4364 Esch-sur-Alzette, Luxembourg} \email{chenglj@zjut.edu.cn
 % \text{\rm and} lijuan.cheng@uni.lu }
 \author[A. Thalmaier]{Anton Thalmaier} \address{Department of
   Mathematics, University of Luxembourg, L--4364 Esch-sur-Alzette,
   Luxembourg} \email{anton.thalmaier@uni.lu}

  \date{\today}

 \begin{abstract}
   Let $M$ be a differentiable manifold endowed with a family of
   complete Riemannian metrics $g(t)$ evolving under a geometric flow
   over the time interval $[0,T[$.  In this article, we give a
   probabilistic representation for the derivative of the
   corresponding conjugate semigroup on $M$ which is generated by a
   Schr\"{o}dinger type operator.  With the help of this derivative
   formula, we derive fundamental Harnack type inequalities in the
   setting of evolving Riemannian manifolds.  In particular, we
   establish a dimension-free Harnack inequality and show how it can
   be used to achieve heat kernel upper bounds in the setting of
   moving metrics. Moreover, by means of the supercontractivity of the
   conjugate semigroup, we obtain a family of canonical log-Sobolev
   inequalities.  We discuss and apply these results both in the case
   of the so-called modified Ricci flow and in the case of general
   geometric flows.
 \end{abstract}

 \maketitle

\section{Introduction}\label{introduction-section}
Let $M$ be a differentiable manifold endowed with a $C^1$ family of
complete Riemannian metrics $g(t)$ indexed by the real interval
$[0,T[$ where $T\in{]0,\infty]}$. The family $g(t)$ describes the
evolution of the manifold $M$ under a geometric flow where $T$ is the
first time where possibly a blow-up of the curvature occurs.  This
type of singularities is not excluded in our setting.

More specifically, we consider geometric flows of the following type:
\begin{align*}
  \partial_tg(t)=-2h(t),\quad \text{on }M\times [0,T[,
\end{align*}
where $h(t)$ is a general time-dependent symmetric $(0,2)$-tensor.
For fixed $t$, with respect to the metric $g(t)$, let
$\sH_t=\tr_{g(t)}h(t)$ be the metric trace of the tensor $h(t)$ and
$\Delta_t$ the Laplace-Beltrami operator acting on functions on
$M$. In practice, the geometric flow deforms the geometry of $M$ and
smoothens out irregularities in the metric to give it a nicer and more
symmetric form which provides geometric and topological information on
the manifold.

Consider operators of the form $L_t=\Delta_t-\nabla^t\phi_t$ where
$\phi_t$ is a time-dependent function on $M$.  We also use the
notation $g_t=g(t)$ and $h_t=h(t)$. In this paper we study the
(minimal) fundamental solution to heat equations of the type
\begin{align*}
  (L_t-\partial_t)u(t,x)=0,\quad \text{resp.}\quad (L_t+\partial_t-\varrho_t)u(t,x)=0,
\end{align*}
where $\varrho_t=\partial_t\phi_t+\sH_t$. The first equation is the
classical heat equation, the second one appears naturally as conjugate
heat equation.  More precisely, we have the following relationship.

\begin{remark} Set $d\mu_t=\e^{-\phi_t}d\vol_t$ where $\vol_t$ denotes
  the Riemannian volume to the metric $g(t)$.  Let
  $\Box=L_t-\partial_t$ be the standard heat operator and $\Box^*$ its
  formal adjoint with respect to the measure $\mu_t\otimes dt$. Thus,
  \begin{align*}
    \int_0^{T}\int_MV\Box U\, d \mu_t \,dt=\int_0^{T}\int_M U\Box^* V\, d \mu_t \,dt
  \end{align*}
  for smooth functions $U, V\colon M\times {[0,T[}\to {[0,\infty[}$.
  From this relation it is immediate that
  $\Box^*=L_t+\partial_t-\varrho_t$.
\end{remark}

\begin{example}
  A typical situation covered by this setting is solving a geometric
  flow equation (e.g. Ricci flow) forward in time and the associated
  conjugate heat equation backward in time. In the case of the Ricci
  flow $\partial_tg(t)=-2\Ric_t$ and $L_t=\Delta_t$, the conjugate
  heat equation reads as $\Box^*u=(\Delta_t+\partial_t-R)u=0$ where
  $R=\tr\,\Ric$ denotes the (time-dependent) scalar curvature.
\end{example}

It should be mentioned that an important ingredient in the proof of
the Poincar\'{e} conjecture by Perelman is a differential Harnack
inequality which is related to a gradient estimate for solutions to
the conjugate heat equation under forward Ricci flow on a compact
manifold \cite{Pe12}.  This relation has been one of our motivations
to investigate solutions to conjugate heat equations and their
associated semigroups also by methods of stochastic analysis.

Let $X_t$ be the diffusion process generated by
$L_t=\Delta_t-\nabla^t\phi_t$ (called $L_t$-diffusion process) which
we assume to be non-explosive up to time $T$. We consider the
two-parameter semigroup associated to~$L_t$,
\begin{align*}
  P_{s,t}f(x):=\E\l[f(X_t)\,|\,X_s=x\r],\quad s\leq t,
\end{align*}
which satisfies the heat equation
\begin{align*}
  \left\{
  \begin{aligned}
    &\frac{\partial }{\partial s}P_{s,t}f=-L_sP_{s,t}f, \\
    &\lim_{s\rightarrow t}P_{s,t}f=f.
  \end{aligned}
      \right.
\end{align*}

In previous work, we already studied properties of heat equations
under geometric flows, like properties of the semigroup $P_{s,t}$
generated by $L_t$, by adopting probabilistic methods. In \cite{Ch17},
for instance, the first author gave functional inequalities equivalent
to a lower bound of $\Ric_t-\frac{1}{2}\partial_tg_t$.  In
\cite{Cheng-Thalmaier_JGA:2018,Cheng-Thalmaier_JFA:2018} we
established characterizations of upper and lower bounds for
$\Ric_t-\frac{1}{2}\partial_tg_t$ in terms of functional inequalities
on the path space over $M$.

On a more probabilistic side, in \cite{ChT18} the authors studied
existence and uniqueness of so-called evolution systems of measures
related to the semigroup $P_{s,t}$. Using such systems as reference
measures, contractivity of the semigroup, as well as log-Sobolev
inequality, have been investigated.

Although the evolution system of measures is helpful to shed light on
properties of solutions to the heat equation, it is still difficult to
obtain a full picture of this measure system, like its relation to the
system of volume measures. It has been observed that if one uses
volume measures as reference measures, many questions will be related
to the conjugate heat equation and not the usual heat equation, see
e.g. \cite{Ab15,CGT15}.

Recall that $\mu_t(dx)=\e^{-\phi_t(x)}d\vol_t$ where $\vol_t$ is the
volume measure with respect to the metric $g(t)$.  Let
\begin{align*}
  P_{s,t}^{\varrho}f(x)=\E\l[\exp{\l(-\int_s^t\varrho(r,X_r)\, dr\r)}\,f(X_t)\,\Big|\,X_s=x\r],
\end{align*}
where $\varrho(t,x)=\varrho_t(x)$ is given by
\begin{align*}
  \frac{\partial}{\partial t} \mu_t(dx)=-\l(\partial_t\phi_t+\mathcal{H}\r)(t,x)\,\mu_t(dx)
  =-\varrho(t,x)\,\mu_t(dx).
\end{align*}
According to the Feynman-Kac formula, $P_{s,t}^{\varrho}f$ represents
the solution to the equation
\begin{align*}
  \frac{\partial}{\partial s}\varphi_s=-(L_s-\varrho_s)\varphi_s,\quad \varphi_t=f,
\end{align*}
on $[0,t]\times M$ where $t<T$. We note that this equation is
conjugate to the heat equation
\begin{align*}
  \frac{\partial}{\partial s}u(s,x)=L_su(s,\cdot)(x).
\end{align*}

In this paper, we first give probabilistic formulae and estimates for
$dP_{s,t}^\varrho f$ from where we then derive a dimension-free
Harnack inequality. It is interesting to note that by combining the
dimension-free Harnack inequalities for $P_{s,t}$ and
$P_{s,t}^\varrho$, one can obtain new upper bounds for the heat kernel
to $L_t$ with respect to $\mu_t$, see
Theorem~\ref{Heat-kernel-estimate}.  We apply this result then (see
Corollary \ref{cor1} below) to the following modified geometric flow
for $g_t$ combined with the conjugate heat equation for $\phi_t$,
i.e.,
\begin{align}\label{modified-geoflow}
  \left\{
  \begin{aligned}
    \partial_tg_t&=-2(h+\Hess(\phi))_t; \\
    \partial_t \phi_t&=-\Delta_t\phi_t-\sH_t.
  \end{aligned}
                       \right.
\end{align}
As is well-known \cite{Pe12}, when $h_t=\Ric_t$, this flow gives the
gradient flow to Perelman's entropy functional $\scrF$.

We observe that
\begin{align}\label{Eq:mut}
  \mu_s(P_{s,t}^{\varrho}f)=\mu_t(f)
\end{align}
which means that the family of measures $\mu_s$ plays for the
semigroup $P_{s,t}^{\varrho}$ a similar role as the invariant measure
for the one-parameter semigroup $P_t$ on the static Riemannian case.
Log-Sobolev inequalities with respect to the invariant measure are
well established under certain curvature conditions on static
Riemannian manifolds; they are related to other functional
inequalities for $P_t$ and have many applications, see for instance
\cite{Bakry,Gross, Wang01,Wang09}.  This leads to the natural question
whether one can prove log-Sobolev inequalities with respect to $\mu_t$
in a similar way through functional inequalities for
$P_{s,t}^{\varrho}$. In Section \ref{log-sobolev-section}, we discuss
the relation between contraction properties of the semigroup and
log-Sobolev inequalities with respect to $\mu_t$. Using the
dimension-free Harnack inequality for $P_{s,t}^\varrho$, we give a
sufficient condition for supercontractivity of $P_{s,t}^\varrho$ which
we then use to establish existence of the (defective) log-Sobolev
inequality for $\mu_s$.  It is well-known that the log-Sobolev
inequality or Sobolev inequality is an important tool to prove an
upper bound of the heat kernel, see \cite{Ab16,B12, Zh12}.

We prove the following result about system
\eqref{modified-geoflow}. Denote by $\rho_t\equiv\rho_t(o,\newdot)$
the distance function to a given base point $o$ in $M$ with respect to
the metric $g_t$.  Suppose that the geometric flow $g_t$ and and the
function $\phi_t$ satisfy \eqref{modified-geoflow}. Furthermore,
assume that $\Ric_t+\Hess_t(\phi_t)-\frac{1}{2}\partial_tg_t\geq K(t)$
and $\mu_t\big(\exp(\lambda \rho_t^2)\big)<\infty$ for all $\lambda>0$
and $t\in [0,T[$, then there exists a function $\beta$ such that
\begin{align*}
  \mu_s(f^2\log f^2)\leq r\mu_s(|\nabla^s f|^2_s)+\beta_s(r),\quad r>0,
\end{align*}
for $f\in C^{\infty}_0({[0,T[}\times M)$ and $\mu_s(f^2)=1$.

The remaining part of the paper is organized as follows. In Section
\ref{derivative-formula}, a probabilistic formula for the derivative
of the conjugate heat semigroup is given and used to derive a gradient
estimate for $P_{s,t}^{\varrho}$ under suitable curvature
conditions. In Section \ref{Harnack-ineq-section}, we derive two
versions of dimension-free Harnack inequalities from the mentioned
gradient inequality for $P_{s,t}^{\varrho}$, which are then applied in
Section \ref{generalflow-section} to gain a sufficient and necessary
condition for supercontractivity of~$P_{s,t}^{\varrho}$. In this
section we also clarify the relation between the supercontractivity of
$P_{s,t}^{\varrho}$ and the log-Sobolev inequality with respect to
$\mu_t$. These results are applied to system \eqref{modified-geoflow}
of the modified geometric flow under conjugate heat equation.

\section{Brownian motion with respect to evolving manifolds}\label{Sect1}\noindent
Let $(M,g_t)_{t\in I}$ be an evolving manifold indexed by
$I=[0,T[$. Let $\nabla^t$ be the Levi-Civita connection with respect
to $g_t$. Denote by $\M:=M\times I$ space-time and consider the bundle
$$TM\xrightarrow{\pi} \M$$ where $\pi$ is the projection.
As observed by Hamilton~\cite{Hamilton:93} there exists a natural
space-time connection $\nabla$ on $TM$ considered as bundle over
space-time $\M$ such that
$$\left\{
\begin{aligned}
  \nabla_vX&=\nabla^t_vX,\quad\text{and}\\
  \nabla_{\partial_t}X&=\partial_tX+\frac12(\partial_tg_t)(X,\newdot)^{\sharp
    g_t}
\end{aligned}
\right.
$$
for all $v\in (T_xM,g_t)$ and all time-dependent vector fields $X$ on
$M$.  This connection is compatible with the metric, i.e.
$$\frac d{dt}|X|_{g_t}^2=2\langle
X,\nabla_{\partial_t}X\rangle_{g_t}.$$

\begin{remark}
  Let $G=\text{O}(n)$ where $n=\dim M$ and consider the $G$-principal
  bundle $\mathcal{F}\xrightarrow{\pi} \M$ of orthonormal frames with
  fibres
$$\mathcal{F}_{(x,t)}=\left\{u\colon \R^n\to(T_xM,g_t)\mid u\text{ isometry}\right\}.$$
As usual, $a\in G$ acts on $\mathcal{F}$ from the right via
composition.  The connection $\nabla$ gives rise to a $G$-invariant
splitting of the sequence \newcommand{\new}[5]{\begin{center}
    \begin{tikzpicture}[->,>=stealth',shorten >=1pt,auto,node
      distance=2cm,thick]
      \node (1) {$#1$}; \node (2) [right of=1] {$#2$}; \node (3)
      [right of=2] {$#3$}; \node (4) [right of=3] {$#4$}; \node (5)
      [right of=4] {$#5$}; \path (1) edge node [right] {} (2) (2) edge
      node [right] {} (3) (3) edge node (pi) {$d\pi$} (4) (4) edge
      node [right] {} (5); \path[dashed] (4) edge [bend left=45] node
      (phi) {$h$} (3);
    \end{tikzpicture}\end{center}}
\new{0}{\text{ker}\,d\pi}{T\mathcal{F}}{\pi^\ast T\M}{0,} which
induces a decomposition of $T\mathcal{F}$ as
$T\mathcal{F}=V\oplus H:=\ker d\pi\oplus h(\pi^\ast T\M)$.
% $G$-invariance of the splitting means that $H_{u g}=(dR_g)H_u$ for
% each $u\in\mathcal{F}$, where $R_gu:=u\, g$ denotes the right action
% of $g\in G$.
For $u\in \mathcal{F}$, the space $H_u$ is the \textit{horizontal
  space at} $u$ and $V_u=\{v\in T_u\mathcal{F}\colon\ (d\pi)v=0\}$ the
\textit{vertical space at} $u$.  The bundle isomorphism
\begin{equation}
  \label{Eq:horLift}
  h\colon\,\pi^\ast T\M\acong H\rightinj T\mathcal{F},\quad h_u\colon\,T_{\pi(u)}\M\acong H_u,\quad u\in\mathcal{F},
\end{equation}
is the \textit{horizontal lift} of the $G$-connection.
\end{remark}

\begin{corollary}
  To each $a X+b \partial_t\in T_{(x,t)}\M$ and each frame
  $u\in\mathcal{F}_{(x,t)}$, there exists a unique ``horizontal lift''
  $a X^*+b D_t\in H_u$ of $a X+b \partial_t$ such
  that
$$\pi_*(a X^*+b D_t)=a X+b\partial_t.$$
In explicit terms, $X^*$ is the horizontal lift of $X$ with respect to
the metric~$g_t$, and $D_t=\left.\frac d{ds}\right\vert_{s=0}u_s$
where $u_s$ is the horizontal lift based at $u$ of the curve
$s\mapsto (x,t+s)$.
\end{corollary}

We consider curves in $\M$ of the form
$$\gamma_t=(x_t,\ell_t),\quad t\in[0,T[$$
where $\ell_t$ is a monotone differentiable transformation on $[0,T[$.
The horizontal lift of such a curve $\gamma_t$ in $\M$ is a curve
$u_t$ in $\mathcal{F}$ such that $\pi u_t=\gamma_t$ and
$\nabla_{\dot\gamma}(u_te)=0$ for each $e\in\R^n$.  Then
$$/\!/_{r,s}^\gamma:=u_s^{\mstrut}u_r^{-1}\colon(T_{x_r}M,g_{\ell_r}^{\mstrut})\to(T_{x_s}M,g_{\ell_s}^{\mstrut}),\quad
0\leq r\leq s< T,$$ gives parallel transport along $\gamma_t$.  In the
following we consider the special case $\ell_t=t$.

\begin{remark}
  Vector fields and differential forms on $\M$ can be seen as
  time-dependent vector fields and differential forms on $M$.  It is
  convenient to write objects on $\M$ as $G$-equivariant functions
  on~$\mathcal{F}$.  In particular, then
  \begin{enumerate}[\rm1)]
  \item functions $f\in C^\infty(\M)$ read as
    $\tilde f\in C^\infty(\mathcal{F})$ via $\tilde f:=f\circ \pi$;
  \item time-dependent vector fields $Y$ on $M$ read as
    $\tilde Y\colon\mathcal{F}\to\R^n$ via
    $\tilde Y(u):=u^{-1}Y_{\pi(u)}$;
  \item time-dependent differential forms $\alpha$ on $M$ as
    $\tilde\alpha\colon\mathcal{F}\to(\R^n)^*$ via
    $\tilde\alpha(u)=\alpha_{\pi(u)}(u\,\cdot)$.
  \end{enumerate}
  The following formulas hold:
  \begin{align}
    &\widetilde{Xf}=X^*\tilde f,\quad\widetilde{\partial_tf}=D_t\tilde f,\quad
      \widetilde{\nabla_{\!X}^\mstrut Y}=X^*\tilde Y,\quad\widetilde{\nabla_{\!\partial_t}^\mstrut Y}=D_t\tilde Y,\notag\\
    &\widetilde{\nabla_{\!X}^\mstrut\alpha}=X^*\tilde \alpha,\quad\widetilde{\nabla_{\!\partial_t}^\mstrut \alpha}
      =D_t\tilde\alpha.\label{Eq:equiv_repr}
  \end{align}
  The proofs are straightforward. For instance, to check the last
  equality, let $u_t$ be a horizontal curve such that
  $\pi u_t=\gamma_t=(x,t)$ where $x\in M$ is fixed.  Then
  \begin{align*}
    (D_t\tilde \alpha)(u_s)&=\left.\frac d{dt}\right|_{t=s}\tilde \alpha(u_t)
                             =\left.\frac d{dt}\right|_{t=s}\alpha_{\pi(u_t)}(u_t\,\cdot)\\
                           &=\left.\frac d{dr}\right|_{r=0}\alpha_{(x,s+r)}(/\!/_{s,s+r}u_s,\cdot)\\
                           &=(\nabla_{\!\partial_t}^\mstrut \alpha)_{(x,s)}(u_s\,\cdot)
                             =(\widetilde{\nabla_{\!\partial_t}^\mstrut \alpha})(u_s).
  \end{align*}
\end{remark}

\begin{remark}
  The vector fields
  $$H_i\in\Gamma(T\mathcal{F}),\quad H_i(u)=(ue_i)^*\equiv
  h_u(ue_i),\quad i=1,\ldots,n,$$ where $(e_1,\ldots,e_n)$ denotes the
  standard basis of $\R^n$, are the standard-horizontal vector fields
  on $\mathcal{F}$. The operator
$$\Delta_{\text{\rm hor}}=\sum_{i=1}^nH_i^2$$
is called Bochner's horizontal Laplacian on $\mathcal{F}$.  Note that
\begin{align}\label{Eq:horLapl}
  \widetilde{\Delta f}=\Delta_{\text{hor}}\tilde f\quad\text{and}\quad
  \widetilde{\Delta_{\text{rough}} \alpha}=\Delta_{\text{hor}}\tilde \alpha
\end{align}
where $\Delta_{\text{rough}}=\tr(\nabla^t)^2$ is the rough Laplacian
on differential one-forms. Recall that, for fixed $t\in I$,
\begin{align}\label{Eq:Weitzenboeck}
  d\Delta_tf=\tr(\nabla^t)^2df-\Ric_t(df,\newdot)
\end{align}
by the Weitzenb\"{o}ck formula.
\end{remark}

Let $\pi\colon\mathcal{F}\to \M$ denote the canonical projection.  For
$\phi\in C^{1,2}(\M)$, we define a vector field on $\mathcal{F}$ by
$$H^\phi\in\Gamma(T\mathcal{F}),\quad H^\phi(u)=h_u(\nabla^t\phi(t,\newdot)_x),\quad
\pi(u)=(x,t).$$

Consider the following Stratonovich SDE on $\mathcal{F}$:
\begin{align}\label{EQ:SDE_on_frames}\left\{
  \begin{aligned}
    dU&=D_t(U)\,dt+\sum_{i=1}^nH_i(U)\circ dB^i-H^\phi(U)\,dt,\\
    U_s&=u_s,\ \pi(u_s)=(x,s),\ s\in{[0,T[}.
  \end{aligned}
         \right.
\end{align}
Here $B$ denotes standard Brownian motion on $\R^n$ (speeded up by the
factor $2$, i.e., $dB^idB^j=2\delta_{ij}dt$) with generator
$\Delta_{\R^n}$.  Eq.~\eqref{EQ:SDE_on_frames} has a unique solution
up to its lifetime $\zeta:=\lim\limits_{k\rightarrow \infty}\zeta_k$
where
\begin{align}\label{zeta-n}
  \zeta_k:=\inf\left\{t\in [s,T[\colon \rho_t(\pi(U_s), \pi(U_t))\geq k\right\}, \quad n\geq 1,\quad \inf\varnothing:=T,
\end{align}
and where $\rho_t$ stands for the Riemannian distance induced by the
metric $g(t)$.  We note that if $U$ solves \eqref{EQ:SDE_on_frames}
then
$$\pi(U_t)=(X_t,t)$$
where $X$ is a diffusion process on $M$ generated by
$L_t=\Delta_t-\nabla^t\phi_t$.  In case of $\phi=0$ we call $X$ a
$(g_t)$-Brownian motion on $\M$.

More precisely, we have the following result.

\begin{proposition}\label{Prop:ItoFrameBdl}
  Let $U$ be a solution to the
  \textup{SDE}~\eqref{EQ:SDE_on_frames}. Then
  \begin{enumerate}[\rm(1)]
  \item for any $C^2$-function $F\colon\mathcal{F}\to\R$,
$$d(F(U))=(D_t F)(U)\,dt+
\sum_{i=1}^n(H_iF)(U)\,dB^i+ (\Delta_{\text{\rm hor}}F)(U)\,dt-(H^\phi
F)(U)\,dt;$$
\item for any $C^2$-function $f\colon\M\to\R$, we have
$$d(f(X))=(\partial_t f)(X)\,dt+
\sum_{i=1}^n(Ue_if)\,dB^i+ (L_tf)(X)\,dt.$$
\end{enumerate}
\end{proposition}

Let $U$ be a solution to the \textup{SDE}~\eqref{EQ:SDE_on_frames} and
$\pi(U_t)=(X_t,t)$. Furthermore let
$$/\!/_{r,t}^{\mstrut}:=U_t^{\mstrut}U_r^{-1}\colon(T_{x_r}M,g_r^{\mstrut})\to(T_{x_t}M,g_t^{\mstrut}),\quad
s\leq r\leq t<T,$$ be the induced parallel transport along $X_t$
(which by construction consists of isometries).  We use the notation
$$X_t=X_t^{(s,x)},\quad t\geq s,$$
if $X_s=x$.  Note that $X_t=X_t^{(s,x)}$ solves the equation
$$d X_t^{(s,x)}=U_t\circ \,d B_t-\nabla^t\phi_t(X_t^{(s,x)})\,dt,\quad X_s^{(s,x)}=x.$$
For any $f\in C_0^2(M)$,
$$f(X_t^{(s,x)})-f(x)-\int_s^tL_rf(X^{(s,x)}_r)\,d r=\int_s^t\big<\partrinv sr\nabla^rf(X^{(s,x)}_r), U_s\,d B_r\bigr>_s,
\quad t\in [s,T[,$$ is a martingale up to the lifetime $\zeta$.  In
the case $s=0$, we write again $X_t^{x}$ instead of~$X_t^{(0,x)}$.

\section{Derivative formula}\label{derivative-formula}

Let $L_t=\Delta_t-\nabla^t\phi_t$ where $\phi$ is
$C^{1,2}([0,T[\times M)$.  Throughout this section, we assume that the
diffusion $(X_t)$ generated by $L_t$ is non-explosive up to time $T$.
Recall that $\mu_t(dx)=\e^{-\phi_t(x)}d \vol_t$ and
\begin{align*}
  \frac{\partial}{\partial t} \mu_t(dx)
  =-\l(\partial_t\phi+\mathcal{H}\r)(t,x)\,\mu_t(dx)=-\varrho(t,x)\,\mu_t(dx)
\end{align*}
with
$$\varrho(t,x)\equiv\varrho_t(x)=(\partial_t\phi+\mathcal{H})(t,x).$$
For each $t$, we assume that $\varrho_t$ is bounded from below by a
constant depending on $t$.

For $f\in C_b(M)$ let
\begin{align}\label{Eq:Feynman-Kac}
  P_{s,t}^{\varrho}f(x)=\E\l[\exp\l(-\int_s^t\varrho(r,X_r)\, dr\r)f(X_t)\,\big|\,X_s=x\r].
\end{align}
Then
\begin{align}\label{Eq:HeatEqu}
  \frac{\partial}{\partial s}P_{s,t}^{\varrho}f=-(L_s-\varrho_s)P_{s,t}^{\varrho}f,\quad
  P_{t,t}^{\varrho}=f.
\end{align}
That $u(s,x):=P_{s,t}^{\varrho}f(x)$ represents the solution to
\eqref{Eq:HeatEqu} is easily seen from the fact that
$$\exp\l(-\int_s^r\varrho(a,X_a)\, da\r)P_{r,t}^{\varrho}f(X_r^{(s,x)}),\quad r\in [s,t],$$
is a martingale under given assumptions.

Our first step is to establish a derivative formula for
$P_{s,t}^{\varrho}$.  When the metric is static, the derivative
formula for $P_tf$ is known as Bismut formula \cite{Bismut84,EL94}.
The more general versions in \cite{Thal97} have been adapted to
Feynman-Kac semigroups in \cite{Th19}.

We fix $s\in {[0,T[}$ and consider the random family
$Q_{s,t}\in\Aut(T_{X_s}M)$, $0\leq s\leq t< T$, given as solution to
the following ordinary differential equation:
\begin{align}\label{eq-Q3}
  \frac{d Q_{s,t}}{d t}=-\l(\Ric-\frac{1}{2}\partial g+\Hess(\phi)\r)_{\partr st}\!Q_{s,t},\quad Q_{s,s}=\id,
\end{align}
where
\begin{align*} {\l(\Ric-\frac{1}{2}\partial g+\Hess(\phi)\r)}_{\partr
    st}:=\partrinv st\circ
  {\l(\Ric_t-\frac{1}{2}\partial_tg_t+\Hess_t(\phi_t)\r)}(X_t)\circ
  \partr st.
\end{align*}

\begin{theorem}\label{Der-formula}
  Let $L_t=\Delta_t-\nabla^t\phi_t$ as above. For each $t$, assume
  that both $\varrho_t$ and
  $$\big(\Ric+\Hess(\phi)-\frac{1}{2}\partial_tg\big)(t,\cdot)$$ are bounded
  below and that $|d\varrho_t|$ is bounded.  Then, for $v\in T_xM$ and
  $f\in C^1_b(M)$,
  \begin{align}
    (dP_{s,t}^{\varrho}f)(v)
    &=\E^{(s,x)}\l[\exp\l(-\int_s^t\varrho_r(X_r)\, d r\r)\right.\notag\\
    &\quad\times\l.\l(df(\partr st Q_{s,t}v)(X_t)- f(X_t)\int_s^t d\varrho _r(\partr sr Q_{s,r}v)\,d r\r)\r].
      \label{Eq:Der-formula}
  \end{align}
\end{theorem}

\begin{proof}
  By the definition of $Q$ as the solution to \eqref{eq-Q3}, we have
  \begin{align}\label{Eq:defQrs}
    d(U_s^{-1}Q_{s,r})=-U_s^{-1}\big(\Ric+\Hess(\phi)\big)_{\partr sr}Q_{s,r}.
  \end{align}
  Set
  $$N_r(v)=d P_{r,t}^{\varrho}f(\partr sr Q_{s,r}v),\quad s\leq r\leq
  t.$$ Write $N_r(v)=F(U_r,q_r(v))$ where
  \begin{align*}
    F(u,w)&:=(d P_{r,t}^{\varrho}f)_{x}(uw), \quad \pi(u)=(x,r)\text{ and }w\in \R^n,\\
    q_r(v)&:=U_s^{-1}Q_{s,r}v.
  \end{align*}
  By means of Proposition \ref{Prop:ItoFrameBdl}, we have
  \begin{align}\label{Eq:ItoF}
    d(F(U_r,w))\mequal(D_r F)(U_r,w)\,dr+
    (\Delta_{\text{\rm hor}}F)(U_r,w)\,dr-(H^\phi F)(U_r,w)\,dr
  \end{align}
  where
  \begin{align*}
    (D_r F)(u,w)=\partial_r(d P_{r,t}^{\varrho}f)_x(uw)
    +\frac12(\partial_rg_r)\l((d P_{r,t}^{\varrho}f)^{\sharp_{g_r}},uw\r)
  \end{align*}
  and where $\mequal$ stands for equality modulo the differential of a
  local martingale.  Using the Weitzen\-b\"{o}ck formula we observe
  that
  \begin{align*}
    \partial_r(d P_{r,t}^{\varrho}f)
    &=-d(L_r-\varrho_r)P_{r,t}^{\varrho}f\\
    &=-d(\Delta_r-\nabla^r\phi_r-\varrho_r)P_{r,t}^{\varrho}f\\
    &=-\tr(\nabla^t)^2d P_{r,t}^{\varrho}f
      +d P_{r,t}^{\varrho}f\l((\Hess\phi_r)^{\sharp_{g_r}}\r)
      +d P_{r,t}^{\varrho}f(\Ric_r^{\sharp_{g_r}})
      +d\l(\varrho_r\,P^{\varrho}_{r,t}f\r)_{X_r}.
  \end{align*}
  Taking \eqref{Eq:horLapl} and \eqref{Eq:equiv_repr} into account, we
  have
  \begin{align*}
    (\Delta_{\text{\rm hor}}F)(U_r,w)=\tr(\nabla^r)^2d P_{r,t}^{\varrho}f(U_rw)
  \end{align*}
  and
  \begin{align*}
    (H^\phi F)(U_r,w)=\Hess(\phi_r)\l(d P_{r,t}^{\varrho}f)^{\sharp_{g_r}},U_rw\r)=
    d P_{r,t}^{\varrho}f\l((\Hess(\phi_r))^{\sharp_{g_r}}\r)(U_rw).
  \end{align*}
  Thus Eq.~\eqref{Eq:ItoF} rewrites as
  \begin{align}\label{Eq:ItoFnew}
    d(F(U_r,w))\mequal
    d P_{r,t}^{\varrho}f\,(\Ric_r^{\sharp_{g_r}})U_rw
    +d\l(\varrho_r\,P^{\varrho}_{r,t}f\r)_{X_r}U_rw+\frac12(\partial_rg_r)\l((d P_{r,t}^{\varrho}f)^{\sharp_{g_r}},U_rw\r).
  \end{align}
  Combining Eqs.~\eqref{Eq:ItoFnew} and \eqref{Eq:defQrs} we thus
  obtain
  \begin{align*}
    d N_r(v)&=d\l(F(U_r,\newdot)\r)(q_r(v))+F\l(U_r,\partial_rq_r(v)\r)dr\\
            &\mequal d\l(\varrho_r\,P^{\varrho}_{r,t}f\r)_{X_r}\partr sr Q_{s,r} v\,dr\\
            &=\l(\varrho_r(X_r)d P_{r,t}^{\varrho}f(\partr sr Q_{s,r} v)
              +P_{r,t}^{\varrho}f(X_r)d\varrho_r(\partr sr Q_{s,r}v)\r)dr,
  \end{align*}
  Hence we get
  \begin{align*}
    d\l(\exp\l(-\int_s^r\varrho_u(X_u)\, d u\r)N_r(v)\r)
    \mequal -\exp\l(-\int_s^r\varrho_u(X_u)\, d u\r)P^{\varrho}_{r,t}f(X_r)d\varrho_r(\partr sr Q_{s,r}v)\,dr.
  \end{align*}
  Integrating this equality from $s$ to $t$ and taking expectation
  gives Eq.~\eqref{Eq:Der-formula}.
\end{proof}

\begin{corollary}\label{est-gradient-ineq}
  Suppose that $\varrho_t$ is bounded below for each $t$, and assume
  that there are functions $\kappa,K\in C([0,T[)$ such that
  $|d \varrho_t|\leq\kappa(t)$, respectively
  \begin{align*}
    \Ric_t-\frac{1}{2}\partial_tg_t+\Hess_t(\phi_t)\geq K(t).
  \end{align*}
  Then, for $f\in C^{\infty}_0(M)$ and $f\geq 0$,
  \begin{align*}
    \l|\nabla^s P_{s,t}^{\varrho}f\r|_s\leq \exp\l(-\int_s^tK(r)\,dr\r)P_{s,t}^{\varrho}|\nabla ^tf|_t+P_{s,t}^{\varrho}f \int_s^t \kappa(r)\exp\l(-\int_s^rK(u)\,du\r)dr.
  \end{align*}
\end{corollary}

\begin{proof}
  The condition
  $\Ric_t-\frac{1}{2}\partial_tg_t+\Hess_t(\phi_t)\geq K(t)$ implies
  that
$$\|Q_{s,t}\|\leq \exp\l(-\int_s^tK(r)\,dr\r).$$
The inequality follows thus from the bound
$|d \varrho_t|\leq\kappa(t)$.
\end{proof}

\section{Dimension-free Harnack inequalities}\label{Harnack-ineq-section}
We first derive two Harnack type inequalities for
$P_{s,t}^{\varrho}$. Such dimension-free Harnack inequalities have
been studied first by Wang, see e.g.~\cite{Wbook14} for more results
in this direction.  We denote by $\mathcal{B}_b(M)$ the space of
bounded measurable functions on $M$.

\begin{theorem}\label{Harnack-ineq}
  Suppose that $\varrho$ is bounded below,
  $|d \varrho_t|\leq\kappa(t)$, and
  \begin{align*}
    \Ric_t-\frac{1}{2}\partial_tg_t+\Hess_t(\phi_t)\geq K(t).
  \end{align*}
  The following two Harnack type inequalities hold for any $p>1$.
  \begin{enumerate}[\rm(i)]
  \item For $0\leq s\leq t<T$ and any non-negative function
    $f\in\mathcal{B}_b(M)$,
    \begin{align}
      (P_{s,t}^{\varrho}f)^p(x)
      &\leq (P_{s,t}^{\varrho}f^p)(y)\notag\\
      &\quad\times\exp{\l((p-1)\int_s^t\sup\varrho_r^-\,d r
        +\frac{p\rho_s^2(x,y)}{4(p-1)\alpha(s,t)}+\frac{p\eta(s,t)\rho_s(x,y)}{\alpha(s,t)}\r)},
        \label{Harnack-ineq1}
    \end{align}
    where
    \begin{align*}
      &\alpha(s,t):=\int_s^t\exp\l(2\int_s^rK(u)\,d u\r)\,d r;\\
      &\eta(s,t):=\int_s^t\int_s^v\kappa(r)\exp\l(2\int_s^vK(u)\,d
        u-\int_s^rK(u)\,d u\r) \, d r\,d v.
    \end{align*}
  \item For $0\leq s\leq t<T$ and any non-negative function
    $f\in\mathcal{B}_b(M)$,
    \begin{align*}
      (P_{s,t}^{\varrho}f)^{p}(x)
      &\leq
        (P_{s,t}^{\varrho}f^{p})(y)\,\E^y\left[\exp\l(-(p-1)\int_s^t\varrho_r(X_r)\,d r\right)\right]\\
      &\qquad\times\exp\l(\frac{p\rho_s^2(x,y)}{4(p-1)\alpha(s,t)}+\frac{2p \eta(s,t)}{\alpha(s,t)}\rho_s(x,y)\r).
    \end{align*}
  \end{enumerate}

\end{theorem}
\begin{proof}
  To facilitate the notion we restrict ourselves to the case $s=0$.
  By approximation and the monotone class theorem, we may assume that
  $f\in C^2(M)$ of compact support and $\inf_M f > 0$.  Given
  $x \neq y$ and $t>0$, let $\gamma\colon [0,t]\rightarrow M$ be the
  constant speed $g_0$-geodesics from $x$ to $y$ of length
  $\rho_0(x,y)$. Let $\nu_s=d \gamma_s/d s$. Thus we have
  $|\nu_s|_0=\rho_0(x,y)/t$. Let
$$h(s)=t\,\frac{\int_0^s\exp\l(2\int_0^rK(u)\,d u\r)d r}{\int_0^t\exp\l(2\int_0^rK(u)\,d u\r)d r}.$$
Then $h(0)=0$ and $h(t)=t$. Let $y_s=\gamma_{h(s)}$ and
\begin{align*}\varphi(s)&=\log \E^{y_s}\l(\exp{\l(-\int_0^{s}\varrho_r(X_r)\, dr\r)}P^{\varrho}_{s,t}f(X_{s})\r)^p\\
                        &=\log P_{0,s}^{p\varrho}(P^{\varrho}_{s,t}f)^p(y_s),
                          \quad s\in [0,t].
\end{align*}
To determine $\varphi'(s)$, we first note that
\begin{align*}
  d(P_{s,t}^{\varrho}f(X_s))^p
  &=d M_s+\l((L_s+\partial_s)(P^{\varrho}_{s,t}f)^p(X_s)\r)d s\\
  &=d M_s+ p(p-1)(P_{s,t}^{\varrho}f)^{p-2}(X_s)\,|\nabla^sP^{\varrho}_{s,t}f|_s^2(X_s)\,ds \\ &\quad+p\varrho_s(X_s)(P^{\varrho}_{s,t}f)^p(X_s)\,d s, \quad  s\leq\zeta_k,
\end{align*}
where $M_s$ is the local martingale part of
$(P_{s,t}^{\varrho}f)^p(X_s)$. This implies
\begin{align*}
  &\E^{x}\l(\exp\l(-\int_0^{s\wedge\zeta_k}\varrho_r(X_r)\, dr\r)P^{\varrho}_{s\wedge\zeta_k,t}f(X_{s\wedge\zeta_k})\r)^p-(P_{0,t}^{\varrho}f)^p(x)\\
  &\quad \geq p(p-1)\E^x\l[\int_0^{s\wedge \zeta_k}\exp{\l(-p\int_0^{u}\varrho_r(X_r)\, dr\r)}(P_{u,t}^{\varrho}f)^{p-2}(X_u)|\nabla^uP_{u,t}^{\varrho}f|_u^2(X_u)\, d u\r].
\end{align*}
Since $\inf_M f>0$, by letting $k\rightarrow \infty$, we deduce that
\begin{align}
  \E^{x}&\l(\exp\l(-\int_0^{s}\varrho_r(X_r)\, dr\r)P^{\varrho}_{s,t}f(X_{s})\r)^p-(P_{0,t}^{\varrho}f)^p(x)\notag \\
        &\geq p(p-1)\int_0^{s}\mathbb{E}^x\l[\exp{\l(-p\int_0^{u}\varrho_r(X_r)\, dr\r)}
          (P_{u,t}^{\varrho}f)^{p-2}|\nabla^uP_{u,t}^{\varrho}f|_u^2(X_u)\r]\,d u.\label{Ito-eq}
\end{align}
Hence, for each $x\in M$,
\begin{align*}
  &\frac{\partial}{\partial s}\E^{x}\l(\exp\l(-\int_0^{s}\varrho_r(X_r)\, dr\r)P^{\varrho}_{s,t}f(X_{s})\r)^p\\
  &\quad\geq p(p-1)\mathbb{E}^x\l[\exp{\l(-p\int_0^{s}\varrho_r(X_r)\, dr\r)}(P_{s,t}^{\varrho}f)^{p-2}|\nabla^sP_{s,t}^{\varrho}f|_s^2(X_s)\r].
\end{align*}
Moreover, by adopting Corollary \ref{est-gradient-ineq} for
$P_{s,t}^{p\varrho}$, i.e.
$$|\nabla^s P_{s,t}^{p\varrho}f|_s\leq \exp\l(-\int_s^tK(r)\,d r\r)P_{s,t}^{p\varrho}\,|\nabla ^tf|_t
+pP_{s,t}^{p\varrho}f \int_s^t \kappa(r)\exp\l(-\int_s^rK(u)\,d
u\r)dr,$$ we thus obtain for $s\in [0,t]$,
\begin{align*}
  \frac{d \varphi(s)}{d s}
  &\geq \left(\frac{1}{P_{0,s}^{p\varrho}(P_{s,t}^{\varrho}f)^p}\Bigg\{p(p-1)P_{0,s}^{p\varrho}\l((P_{s,t}^{\varrho}f)^p\,|\nabla^s\log
    P_{s,t}^{\varrho}f|^2_s\r)+h'(s)\l<\nabla^0P_{0,s}^{p\varrho}(P_{s,t}^{\varrho}f)^{p},\nu_s\r>_0\Bigg\}\right)(y_s)\\
  &\geq \left(\rule{0em}{1.7em}
    \frac{p}{P_{0,s}^{p\varrho}(P_{s,t}^{\varrho}f)^p}\Bigg\{(p-1)P_{0,s}^{p\varrho}\l((P_{s,t}^{\varrho}f)^{p-2}\,|\nabla^sP_{s,t}^{\varrho}f|^2_s\r)\right.\\
  &\qquad-\frac{\rho_0(x,y)}{t} \exp\l(-\int_0^sK(u)\,d u\r)
    h'(s)P_{0,s}^{p\varrho}\l((P_{s,t}^{\varrho}f)^{p-1}|\nabla^s P_{s,t}^{\varrho}f|_s\r)\\
  &\qquad
    -\frac{\rho_0(x,y)}{t}h'(s)\,P_{0,s}^{p\varrho}\,(P_{s,t}^{\varrho}f)^p \left.\int_0^s \kappa(r)\exp\l(-\int_0^rK(u)\,d u\r) \,d r\Bigg\}\rule{0em}{1.7em}\right)(y_s)\\
  &\geq\left(\rule{0em}{1.7em}
    \frac{p}{P^{p\varrho}_{0,s}(P^{\varrho}_{s,t}f)^p}P^{p\varrho}_{0,s}\Bigg\{(P^{\varrho}_{s,t}f)^p\Bigg((p-1)|\nabla^s\log
    P^{\varrho}_{s,t}f|^2_s\right.\\
  &\qquad-\frac{\rho_0(x,y)}{t}h'(s)\exp\l(-\int_0^s K(u)\, d u\r)|\nabla^s\log
    P^{\varrho}_{s,t}f|_s \\
  &\qquad-\frac{\rho_0(x,y)}{t}h'(s) \left.\int_0^s \kappa(r)\exp\l(-\int_0^rK(u)\,d u\r)\, d r\Bigg)\Bigg\}
    \rule{0em}{1.7em}\right)(y_s)\\
  &\geq \frac{-p h'(s)^2}{4(p-1)t^2}\exp\l(-2\int_0^sK(u)\,d u\r)\rho_0(x,y)^2
  \\
  &\qquad-\frac pt h'(s) \int_0^s \kappa(r)\exp\l(-\int_0^rK(u)\,d u\r)\,d r\,\rho_0(x,y),
\end{align*}
where the last inequality follows from the simple fact that
$$aA^2+bA\geq-\frac{b^2}{4a}, \quad a>0.$$
Since
$$h'(s)=\frac{t\exp\left(2\displaystyle\int_0^sK(u)\,d u\right)}
{\displaystyle\int_0^t\exp\left(2\int_0^rK(u)\,d u\right)\,d r},$$ we
thus arrive at
\begin{align*}
  \frac{d \varphi(s)}{ds}
  &\geq -\frac{p\exp\l(\displaystyle\int_0^s2K(u)\,d u\r)}{4(p-1)
    \l(\displaystyle\int_0^t\exp\l(2\int_0^rK(u)\,d u\r)\,d r\r)^2}\,\rho_0(x,y)^2\\
  &\quad-\frac{\displaystyle p\exp\l(2\int_0^sK(u)\,d u\r)
    \int_0^s \kappa(r)\exp\l(\displaystyle-\int_0^rK(u)\,d u\r) \,d r}{\displaystyle\int_0^t\exp\l(2\int_0^rK(u)d u\r)\,d r}
    \,\rho_0(x,y),\quad s\in [0,t].
\end{align*}
Integrating over $s$ from 0 and $t$, we get
\begin{align}
  (P_{0,t}^{\varrho}f)^p(x)&\leq \E^y\l[\l(\exp\l(-\int_0^t\varrho_r(X_r)\,d r\r)f(X_t)\r)^p\r]\notag\\
                           &\quad\times\exp\l(\frac{p\rho_0(x,y)^2
                             }{4(p-1)\displaystyle\int_0^t\exp\l(2\int_0^rK(u)\,d u\r)\,d r}\r. \notag \\
                           &\qquad\qquad+\l.\frac{\displaystyle p\int_0^t\int_0^s\kappa(r)\exp\l(2\int_0^sK(u)d
                             u-\int_0^rK(u)\,d u\r) \,d r\,ds}{\displaystyle\int_0^t\exp\l(2\int_0^rK(u)\,d u\r)\,d r}\,\rho_0(x,y)\r)\notag\\
                           &\leq   \E^y\l[\exp\l(-\int_0^t\varrho_r(X_r)\,d r\r)f(X_t)^p\r]\exp\l((p-1)\int_0^t\sup\varrho_r^-\,d r\r)\notag\\
                           &\quad\times\exp\l(\frac{p\rho_0(x,y)^2
                             }{4(p-1)\displaystyle\int_0^t\exp\l(2\int_0^rK(u)\,d u\r)\,d r}\r. \notag \\
                           &\qquad\qquad+\l.\frac{\displaystyle p\int_0^t\int_0^s\kappa(r)\exp\l(2\int_0^sK(u)\,d
                             u-\int_0^rK(u)\,d u\r)\,d r\,ds}{\displaystyle\int_0^t\exp\l(2\int_0^rK(u)\,d u\r)\,d r}\,\rho_0(x,y)\r). \label{harnack-ineq-1}
\end{align}
This proves part (i) of the theorem. In addition, by adopting in
\eqref{harnack-ineq-1} the estimate
\begin{align*}
  \E^y\l[\l(\exp\l(-\int_0^t\varrho_r(X_r)\,d r\r)f(X_t)\r)^p\r]
  \leq  (P_{0,t}^{\varrho}f^{p})(y)\,\E^y\l[\exp\l(-(p-1)\int_0^t\varrho_r(X_r)\,d r\r)\r],
\end{align*}
part (ii) of the theorem follows as well.
\end{proof}

\section{Heat kernel estimate}\label{heat-kernel-section}
Let $p$ be the fundamental solution of $L_t=\Delta_t-\nabla^t\phi_t$
in the sense that
\begin{align*}
  \left\{
  \begin{aligned}
    &\partial_sp(s,x;r,y)=-(\Delta_{s}-\nabla^s\phi_s)p(s,\newdot;r,y)(x), \\
    &\displaystyle\lim_{s\rightarrow r}p(s,x;r,y)=\delta_y(x).
  \end{aligned}
      \right.
\end{align*}
Note that $p$ is the density of $P_{s,t}(x,dy)$ with respect to
$\mu_t(dy)$, i.e.
\begin{align*}
  P_{s,t}f(x)&=\int p(s,x;t,y)f(y)\,\mu_t(dy)\\
             &=\int p(s,x;t,y)f(y)\e^{-\phi_t(y)}\vol_t(dy),
\end{align*}
where $\vol_t$ denotes the volume measure with respect to $g_t$. In
\cite{Ch17} the following Harnack inequality for the 2-parameter
semigroup $P_{s,t}$ has been derived; it can be seen as a special case
to Theorem \ref{lem1}.\goodbreak

\begin{theorem}\label{lem1}
  Suppose that
$$\Ric_t-\frac{1}{2}\partial_tg_t+\Hess_t(\phi_t)\geq K(t).$$
Then for any non-negative function $f\in \mathcal{B}_b(M)$ and
$0\leq s\leq t<T$,
\begin{align*}
  (P_{s,t}f)^p(x)\leq P_{s,t}f^p(y)\exp\l(\frac{p}{4(p-1)\alpha(s,t)}\rho_s^2(x,y)\r)
\end{align*}
where
\begin{align*}
  \alpha(s,t):=\int_s^t\exp{\l(2\int_s^rK(u)\,d u\r)}\, d r.
\end{align*}
\end{theorem}
It is interesting to observe that Theorem \ref{lem1} for $P_{s,t}$,
along with the Harnack inequality \eqref{Harnack-ineq1} for
$P_{s,t}^{\varrho}$, will allow us to attain upper bounds for the heat
kernel $p$.

\begin{theorem}\label{Heat-kernel-estimate}
  Suppose that, for all $t\in [0,T[$,
  \begin{align*}
    \Ric_t-\frac{1}{2}\partial_tg_t+\Hess_t(\phi_t)\geq K_1(t), \quad \Ric_t+\frac{1}{2}\partial_tg_t+\Hess_t(\phi_t)\geq K_2(t).
  \end{align*}
  Furthermore suppose that $|d\varrho_t|\leq\kappa(t)<\infty$ and
  $\sup\varrho_t^+<\infty$ for each $t\in [0,T[$.
  Then the following heat kernel upper bound holds:
  \begin{align*}
    p(0,x;t,y)\leq \frac{\displaystyle\exp{\l(\frac{1}{2}\int_{0}^{t}\sup \varrho^+_u\, du+\frac{t}{4\alpha_1(0,t/2)}+\frac{t+4\eta_2(t/2,t)\sqrt{t}}{4\alpha_2(t/2,t)}\r)}}{\sqrt{\mu_0\l(B_0(x,\sqrt{t})\r)\mu_{t}\l(B_{t}(y,\sqrt{t})\r)}},
  \end{align*}
  where
  \begin{align}\label{Def:A123}\begin{split}
      &\alpha_1(0,t/2):=\int_0^{t/2}\exp\l(2\int_0^rK_1(u)\,d u\r)d r;\\
      &\alpha_2(t/2,t):=\int_{t/2}^{t}\exp\l(2\int_{t/2}^rK_2(u)\,d u\r)\,d r;\\
      &\eta_2(t/2,t):=\int_{t/2}^t\int_{t/2}^v\kappa(r)\exp\l(2\int_{t/2}^vK_2(u)\,d
      u-\int_{t/2}^rK_2(u)\,du\r) \, drdv.
    \end{split}
  \end{align}
\end{theorem}

\begin{proof}
  We first observe that
  \begin{align*}
    p(0,x;t,y)
    &=\int_M p\l(0,x; t/2,z\r)\,p\l(t/2,z; t,y\r)\,\mu_{t/2}(dz)\\
    &\leq \l(\int_M p\l(0,x; t/2, z\r)^2\mu_{t/2}(dz)\r)^{1/2}\l(\int_Mp\l(t/2,z; t,y\r)^2\mu_{t/2}(d z)\r)^{1/2}.
  \end{align*}
  Hence we are left with the task to estimate the following two terms:
  \begin{align*}
    &I_1=\int_M p\l(0,x; t/2, z\r)^2\,\mu_{t/2}(dz),\\
    &I_2=\int_Mp\l(t/2,z; t,y\r)^2\,\mu_{t/2}(d z).
  \end{align*}
  In order to estimate $I_1$ we proceed with Theorem \ref{lem1}. As
  \begin{align*}
    \Ric_t-\frac{1}{2}\partial_tg_t+\Hess_t(\phi_t)\geq K_1(t)
  \end{align*}
  for some $K_1\in C([0,T[)$, we get
  \begin{align*}
    &(P_{s,t}f)^2(x)\,\mu_s\l(B_s(x,\sqrt{2(t-s)})\r)\,\exp\left(-
      \frac{t-s}{\int_s^t\exp\l(2\int_s^rK_1(u)\,du\r)\,dr}\right)\\
    &\leq \int_M(P_{s,t}f)^2(x)\exp\l(-\frac{\rho_s^2(x,y)}{2\int_s^t\exp\l(2\int_s^rK_1(u)\,du\r)dr}\r)\,\mu_s(dy)\\
    &\leq
      \int_M (P_{s,t}f^{2})(y)\,\mu_s(dy)\\
    &\leq
      \exp\l(\int_s^t\sup \varrho^+(u,\cdot)\, du\r) \int_M f(y)^2\,\mu_t(dy).
  \end{align*}
  Taking
  \begin{align*}
    f(y):=\big(k\wedge p(s,x;t,y)\big),\quad y\in M,\ k\in\N,
  \end{align*}
  we obtain
  \begin{align*}
    \int_M\big(k\wedge p(s,x;t,y)\big)^2\,\mu_t(dy)\leq \frac{\exp\l(\displaystyle
    \int_s^t\sup \varrho^+(u,\cdot)\, du+\alpha_1(s,t)^{-1}(t-s)\r)}{\mu_s\l(B_s(x,\sqrt{2(t-s)})\r)},
  \end{align*}
  where
$$\alpha_1(s,t)=\int_s^t\exp\l(2\int_s^rK_1(u)\,du\r)\,dr.$$
Letting $k\rightarrow \infty$, we arrive at
\begin{align*}
  \int_M p(s,x; t,y)^2\,\mu_t(dy)\leq \frac{\exp\l(\displaystyle
  \int_s^t\sup \varrho^+(u,\cdot)\, du+\frac{t-s}{\alpha_1(s,t)}\r)}{\mu_s\l(B_s(x,\sqrt{2(t-s)})\r)}.
\end{align*}
This shows that
\begin{align}\label{Est:I1}
  I_1\leq \frac{\exp\l(\displaystyle
  \int_0^{t/2}\sup \varrho^+(u,\cdot)\, du+\frac{t}{2\alpha_1(0,t/2)}\r)}{\mu_0\l(B_0(x,\sqrt{t})\r)}.
\end{align}

To estimate the second term $I_2$, we write
\begin{align*}
  \int_M p\l(t/2,z; t, y\r)^2\mu_{t/2}(dz)=\int_Mp_{t}^*\l(0,y; t/2,z\r)^2\mu_{t/2}(d z)
\end{align*}
where $p_t^*(s,x;r,y):=p^*(t-s,x;t-r,y)$ and where $p^*$ denotes the
adjoint heat kernel to $p$. Recall that
$$\left\{
    \begin{aligned}
      &\partial_sp^*_t(s,x;r,y)=-(L_{t-s}+\varrho_{t-s})p^*_t(s,\newdot;r,y)(x), \\
      &\displaystyle\lim_{s\rightarrow r}p^*_t(s,x;r,y)=\delta_y(x).
    \end{aligned}
  \right.
$$
Let $\{\bar{P}_{r,s}^{\varrho}\}_{0\leq r\leq s\leq t}$ be the
semigroup generated by the operator
$L_{t-\bolddot}+\varrho_{t-\bolddot}$. By Theorem \ref{Harnack-ineq1},
this time using relying on the assumption,
$$\Ric_t+\frac{1}{2}\partial_tg_t+\Hess_t(\phi_t)\geq K_2(t),$$
we have
\begin{align*}
  (\bar{P}_{0,t/2}^{\varrho}f)^2(x)\leq (\bar{P}_{0,t/2}^{\varrho}f^2)(y)\exp{\l(\int_0^{t/2}\sup\varrho^+(t-s,\cdot)\,ds +\frac{\rho_t^2(x,y)}{2\alpha_2(t/2,t)}+\frac{2\eta_2(t/2,t)\rho_t(x,y)}{\alpha_2(t/2,t)}\r)},
\end{align*}
where
\begin{align*}
  \alpha_2(t/2,t)&=\int_0^{t/2}\exp\l(2\int_0^rK_2(t-u)\,d u\r)\,d r=\int_{t/2}^{t}\exp\l(2\int_{t/2}^rK_2(u)\,d u\r)\,d r;\\
  \eta_2(t/2,t)&=\int_0^{t/2}\int_0^v\kappa(r)\exp\l(2\int_0^vK_2(t-u)\,d
              u-\int_0^rK_2(t-u)\,du\r) \, drdv\\
            &=\int_{t/2}^t\int_{t/2}^v\kappa(r)\exp\l(2\int_{t/2}^vK_2(u)\,d
              u-\int_{t/2}^rK_2(u)\,du\r)\,drdv.
\end{align*}
By means of this formula, we can proceed as above to obtain
\begin{align*}
  &(\bar{P}_{0,t/2}^{\varrho}f)^2(x)\,\mu_{t}(B_{t}\l(x,\sqrt{t})\r)
    \exp{\l(-\int_0^{t/2}\sup\varrho^+(t-s,\cdot)\,ds -\frac{t}{2\alpha_2(t/2,t)}
    -\frac{2\eta_2(t/2,t)\sqrt{t}}{\alpha_2(t/2,t)}\r)}\\
  &\leq \int_M(\bar{P}_{0,t/2}^{\varrho}f)^2(x)\exp{\l(-\int_0^{t/2}\sup\varrho^+(t-s,\cdot)\,ds
    -\frac{\rho_t^2(x,y)}{2\alpha_2(t/2,t)}
    -\frac{2\eta_2(t/2,t)\rho_t(x,y)}{\alpha_2(t/2,t)}\r)}\,\mu_{t}(dy)\\
  &\leq
    \int_M (\bar{P}_{0,t/2}^{\varrho}f^{2})(y)\,\mu_{t}(dy)= \int_M f(y)^2\,\mu_{t/2}(dy).
\end{align*}
Thus taking $f(z):=p^*_{t}\l(0,y; t/2, z\r)\wedge k$ and letting
$k\rightarrow \infty$, we obtain
\begin{align}\label{Est:I2}
  &\int_M p^*_{t}\l(0,y; t/2, z\r)^2\mu_{t/2}(dz)
    \leq \frac{\exp\l(\displaystyle\int_{t/2}^{t}\sup \varrho^+(u,\cdot)\, du
    +\frac{t+4\eta_2(t/2,t)\sqrt t}{2\alpha_2(t/2,t)}\r)}{\mu_{t}\big(B_{t}(y,\sqrt{t})\big)}.
\end{align}
Finally, combining \eqref{Est:I1} and \eqref{Est:I2} we obtain
\begin{align*}
  p(0,x;t,y)\leq\sqrt{I_1I_2}
  \leq\frac{\displaystyle\exp{\l(\frac12\int_{0}^{t}\sup \varrho^+(u,\cdot)\, du+\frac{t}{4\alpha_1(0,t/2)}+\frac{t+4\eta_2(t/2,t)\sqrt{t}}{4\alpha_2(t/2,t)}\r)}}{\sqrt{\mu_0\big(B_0(x,\sqrt{t})\big)\,\mu_{t}\big(B_{t}(y,\sqrt{t})\big)}}
\end{align*}
where the functions $\alpha_1,\alpha_2,\eta_2$ are defined by \eqref{Def:A123}.
\end{proof}

\begin{remark}\label{rem-heat-kernel}
  In \cite{Coulibaly-Pasquier2019}, the author used a horizontal
  coupling of curves to obtain a dimension-free Harnack inequality for
  $P_{s,t}$ generated by $\Delta_t$, first introduced by Wang, and
  applied it then for an upper bound of the heat kernel. A major
  difference to our approach when $\phi=0$ is that we use the Harnack
  inequality again to deal with the term $I_2$, while in
  \cite{Coulibaly-Pasquier2019} a comparison result for $P_{s,t}$ and
  $P_{s,t}^{\varrho}$ is used so that the condition there includes
  both upper and lower bounds of $\varrho$.

  Gaussian upper bounds for the heat kernel on evolving manifolds have
  recently also obtained by Buzano and Yudowitz
  \cite{2020:Buzano_Yudowitz}.  Their bounds depend on the behaviour
  of the distance function along the flow rather than directly
  involving curvature bounds.
\end{remark}

We now specify our estimates in some specific situations.  First we
consider the modified geometric flow for $g_t$ combined with the
conjugate heat equation for $\phi$, i.e.,
\begin{equation}\label{modified-geoflow-1}
  \left\{
    \begin{aligned} ^{\mathstrut}
      \partial_tg(x,t)&=-2\big(h+\Hess(\phi)\big)(x,t); \\
      \partial_t \phi_t(x)&=-\Delta_t\phi_t(x)-\sH(x,t).
    \end{aligned}
  \right.
\end{equation}
For this modified geometric flow, we have the following result.
\begin{corollary}\label{cor1}
  Suppose that $(g_t, \phi_t)$ evolve by
  Eq.~\eqref{modified-geoflow-1} and that
  \begin{align*}
    \Ric_t-\frac{1}{2}\partial_tg_t+\Hess_t(\phi_t)\geq K_1(t), \quad
    \Ric_t+\frac{1}{2}\partial_tg_t+\Hess_t(\phi_t)\geq K_2(t).
  \end{align*}
  Then
  \begin{align*}
    p(0,x;t,y)\leq \frac{\displaystyle\exp{\l( \frac{t}{4\alpha_1(0,t/2)}+\frac{t}{4\alpha_2(t/2,t)}\r)}}{\sqrt{\mu_0(B_0(x,\sqrt{t}))\mu_{t}(B_{t}(y,\sqrt{t}))}},
  \end{align*}
  where
  \begin{align*}
    \alpha_1(0,t/2)&:=\int_0^{t/2}\exp\l(2\int_0^rK_1(u)\,d u\r)d r;\\
    \alpha_2(t/2,t)&:=\int_{t/2}^{t}\exp\l(2\int_{t/2}^rK_2(u)\,d u\r)\,d r.
  \end{align*}
\end{corollary}

\begin{proof}
  It is immediate from the definition of $\varrho$ that
  \begin{align*}
    \varrho_t&=\partial_t\phi_t+\tr_{g_t}(h_t+\Hess_t(\phi_t))\\
             &=\partial_t\phi_t+\Delta_t\phi_t+\sH_t=0.
  \end{align*}
  The proof is hence completed by applying Theorem
  \ref{Heat-kernel-estimate}.
\end{proof}

We next consider the standard geometric flow for the evolution of the
metric $g$, i.e.,
\begin{align}\label{geoflow-2}
  \left\{
  \begin{array}{ll}
    \partial_tg(x,\cdot)(t)=-2h(x,t); \\
    \phi_t(x)=0.
  \end{array}
  \right.
\end{align}
For this geometric flow, we have $\varrho=\mathcal{H}$ and we thus
obtain the following result.

\begin{corollary}\label{cor-2}
  Suppose that $g_t$ evolves by Eq.~\eqref{geoflow-2}.  Furthermore
  assume that
  \begin{align*}
    \Ric_t-\frac{1}{2}\partial_tg_t\geq K_1(t), \quad \Ric_t+\frac{1}{2}\partial_tg_t\geq K_2(t),
  \end{align*}
  as well as $|d\mathcal{H}_t|\leq\kappa(t)<\infty$ and
  $\sup\mathcal{H}_t^+<\infty$ for each $t\in [0,T[$.  Then
  \begin{align*}
    p(0,x;t,y)\leq \frac{\displaystyle\exp{\l(\frac12\int_{0}^{t}\sup \mathcal{H}^+_u\, du+\frac{t}{4\alpha_1(0,t/2)}+\frac{t+4\eta_2(t/2,t)\sqrt{t}}{4\alpha_2(t/2,t)}\r)}}{\sqrt{\mu_0(B_0(x,\sqrt{t}))\mu_{t}(B_{t}(y,\sqrt{t}))}},
  \end{align*}
  where
  \begin{align*}
    &\alpha_1(0,t/2):=\int_0^{t/2}\exp\l(2\int_0^rK_1(u)d u\r)d r;\\
    &\alpha_2(t/2,t):=\int_{t/2}^{t}\exp\l(2\int_{t/2}^rK_2(u)d u\r)d r;\\
    &\eta_2(t/2,t):=\int_{t/2}^t\int_{t/2}^v\kappa(r)\exp\l(2\int_{t/2}^vK_2(u)\,d
      u-\int_{t/2}^rK_2(u)\,du\r) \, drdv.
  \end{align*}
\end{corollary}

\section{Super log-Sobolev inequalities}\label{log-sobolev-section}
The semigroup $P_{s,t}^{\varrho}$ is called supercontractive if it
maps $L^p(M,\mu_t)$ into $L^q(M,\mu_s)$, i.e.,
\begin{align*}
  \|P_{s,t}^{\varrho}\|_{(p,t)\rightarrow (q,s)}<\infty,
\end{align*}
for any $1<p<q<+\infty$ and $0\leq s \leq t<T$.  In the following
section, we investigate the relation between supercontractivity of
$P_{s,t}^{\varrho}$ and a log-Sobolev inequality with respect to
$\mu_t$.

\subsection{Supercontractivity of $P^{\varrho}_{s,t}$ and Super log-Sobolev inequality }\label{generalflow-section}
We state first the main result of this section.

\begin{theorem}\label{log-S-superbound}
  Assume that $\varrho_t$ is bounded,
  \begin{align*}
    &\Ric_t+\Hess_t(\phi_t)-\frac{1}{2}\partial_tg_t\geq K(t),
      \quad |d \varrho_t|\leq\kappa(t) \quad \mbox{for }\  t\in [0,T[,
  \end{align*}
  and that
$$\|P_{s,t}^{\varrho}\|_{(p,t)\rightarrow (q,s)}< \infty,\quad\text{for $1<p<q\ $ and $\ 0\leq s\leq t<T$.}$$
Then, for every $f\in H^1(M,\mu_t)$ such that $\|f\|_{2,t}=1$,
$t\in [0,T[$, the following super log-Sobolev inequalities hold:
\begin{align}\label{log-S-I}
  \int f^2\log f^2\, d \mu_t\leq  r\, \int |\nabla^tf|_t^2\,d\mu_t+{\beta}_t(r),\quad r>0,
\end{align}
where
\begin{align*}
  \beta_t(r)=\tilde{\beta}_t(\gamma_t^{-1}(r),t)
\end{align*}
and
\begin{align*}
  &\tilde{\beta}_t(s,t)=\frac{pq}{q-p}\log\l(\|P_{s,t}^{\varrho}\|_{(p,t)\rightarrow (q,s)}\r)+\int_{s}^t\l(2  \l(\int_r^t \kappa(u)
    \exp\l(-\int_r^uK(v)\,dv\r) du\r)^2+\sup \varrho_r^{+}\r)\,dr,\\
  &\gamma(s,t)=\frac{4p(q-1)}{q-p}\int_{s}^t\exp\left(-2\int_r^tK(u)\,d u\right)\,d r,\notag\\
  &\gamma^{-1}_t(r)=\inf\l\{s\in[0,t]\colon\gamma(s,t)\leq r\r\}. \notag
\end{align*}
\end{theorem}

The log-Sobolev inequality \eqref{log-S-I} has been shown to be
equivalent to the Sobolev inequality and can hence be used to obtain
an upper bound estimate of the heat kernel, see e.g.~\cite{Zh12} for
details.  In order to prove Theorem \ref{log-S-superbound}, we first
establish a log-Sobolev inequality with respect to the semigroup
$P_{s,t}^{\varrho}$.

\begin{proposition}\label{th1}
  Assume that
  \begin{align*}
    &\Ric_t+\Hess_t(\phi_t)-\frac{1}{2}\partial_tg_t\geq K(t),\ \ \sup|\varrho_t|<\infty\
      \hbox{ and }\ |d \varrho_t|\leq\kappa(t) \ \text{ for }\ t\in [0,T[.
  \end{align*}
  Then, for $0\leq s\leq t< T$ and $f\in C_0^{\infty}(M)$,
  \begin{align*}
    P_{s,t}^{\varrho}(f^2\log f^2)&\leq 4\l(\int_s^t\exp\l(-2\int_r^tK(u)\, d u\r)\,d r\r)P_{s,t}^{\varrho}|\nabla^t f|_t^2+P_{s,t}^{\varrho}f^2\log P_{s,t}^{\varrho}f^2\\
                                  &\quad+\int_s^t\l(2  \l(\int_r^t \kappa(u)\exp\l(-\int_r^uK(v)\,dv\r)\, du\r)^2+\sup \varrho_r^{+}\r)\,dr\, P_{s,t}^{\varrho}f^2.\qedhere
  \end{align*}
\end{proposition}

\begin{proof}
  Without loss of generality, we assume that $f\geq\delta$ for some
  $\delta>0$.  Otherwise, we may take
  $f_{\delta}=(f^2+\delta)^{{1}/{2}}$ and pass to the limit
  $\delta \downarrow 0$ to obtain the conclusion.

  Now consider the process
  $r\mapsto(P_{r,t}f^2)\log (P_{r,t}f^2)(X_{r\wedge \tau_k})$ where as
  above,
  \begin{align}\label{tau-n}
    \tau_k=\inf\big\{t\in [s,T[\colon\rho_t(o,X_t)\geq k\big\},\quad k\geq 1.
  \end{align}
  By means of It\^o's formula, we have
  \begin{align}\label{Ito-pflogpf}
    & d (P_{r,t}^{\varrho}f^2)\log (P_{r,t}^{\varrho}f^2)(X_r)\notag\\
    &=d M_r+(L_r+\partial_r)(P_{r,t}^{\varrho}f^2\log P_{r,t}^{\varrho}f^2)(X_r)\, d r\notag\\
    &=d M_r + \l(\frac1{P_{r,t}^{\varrho}f^2}|\nabla^rP_{r,t}^{\varrho}f^2|_r^2+\varrho_r(1+\log P_{r,t}^{\varrho}f^2)P_{r,t}^{\varrho}f^2\r)(X_r)\, d r,\quad r\leq\tau_k\wedge t,
  \end{align}
  where $M_r$ is a local martingale.  On the other hand, by Corollary
  \ref{est-gradient-ineq}, we have the estimate,
  \begin{align*}
    \l|\nabla^r P_{r,t}^{\varrho}f^2\r|_r
    &\leq \exp\l(-\int_r^tK(u)\,du\r) P_{r,t}^{\varrho}|\nabla ^tf^2|_t+P_{r,t}^{\varrho}f^2 \int_r^t \kappa(u)\exp\l(-\int_r^uK(v)\,dv\r) du\\
    &\leq 2\exp\l(-\int_r^tK(u)\,du\r)P_{r,t}^{\varrho}(f|\nabla ^tf|_t)+P_{r,t}^{\varrho}f^2 \int_r^t \kappa(u)\exp\l(-\int_r^uK(v)\,dv\r) du\\
    &\leq 2\exp\l(-\int_r^tK(u)\,du\r)\sqrt{P_{r,t}^{\varrho}(f^2)P_{r,t}^{\varrho}(|\nabla ^tf|^2_t)}+P_{r,t}^{\varrho}f^2 \int_r^t \kappa(u)\exp\l(-\int_r^uK(v)\,dv\r) du
  \end{align*}
  which gives
  \begin{align*}
    \left|\nabla^r P_{r,t}^{\varrho}f^2\right|_r^2
    &\leq 4\exp\l(-2\int_r^tK(u)\,du\r)(P_{r,t}^{\varrho}f^2)P_{r,t}^{\varrho}(|\nabla ^tf|^2_t)\\
    &\quad  +2(P_{r,t}^{\varrho}f^2)^2 \l(\int_r^t \kappa(u)\exp\l(-\int_r^uK(v)\,dv\r) du\r)^2.
  \end{align*}
  Substituting back into \eqref{Ito-pflogpf}, we obtain
  \begin{align*}
    &d \l(\exp{\l(-\int_s^r \varrho_u(X_u)\,du\r)} (P_{r,t}^{\varrho}f^2)\log (P_{r,t}^{\varrho}f^2)(X_r)\r)\\
    &\leq d M_r+4\exp{\l(-2\int_r^tK(u)\, d u-\int_s^r \varrho_u(X_u)\,du\r)}P_{r,t}^{\varrho}|\nabla^tf|_t^2(X_r)\,d r\\
    &\qquad+2  \l(\int_r^t \kappa(u)\exp\l(-\int_r^uK(v)\,dv\r) du\r)^2\exp{\l(-\int_s^r \varrho_u(X_u)\,du\r)}P_{r,t}^{\varrho}f^2(X_r)\,d r\\
    &\qquad +\varrho_r(X_r)\exp{\l(-\int_s^r \varrho_u(X_u)\,du\r)}P_{r,t}^{\varrho}f^2(X_r)\,d r,\qquad 0\leq  s\leq r\leq \tau_k\wedge t.\end{align*}
  Integrating both sides from $s$ to $t\wedge \tau_k$, taking expectation, and
  letting $k \uparrow +\infty$,
  we obtain by the dominated convergence,
  \begin{align*}
    &P_{s,t}^{\varrho}(f^2\log f^2)-P_{s,t}^{\varrho}f^2\log (P_{s,t}^{\varrho}f^2)\\
    &\ \leq 4\int_s^{t}\exp\l(-2\int_r^tK(u)\,d u\r)\,d r\, P_{s,t}^{\varrho}|\nabla^tf|_t^2\\
    &\qquad +\int_s^t\l[2  \l(\int_r^t \kappa(u)\exp\l(-\int_r^uK(v)\,dv\r) du\r)^2+\sup \varrho_r^{+}\r]\,dr\,P_{s,t}^{\varrho}f^2.
  \end{align*}
  In other words,
  \begin{align}
    P_{s,t}^{\varrho}(f^2\log f^2)
    &\leq 4\l(\int_s^t\exp\l(-2\int_r^tK(u)\,d u\r)\,d r\r)P_{s,t}^{\varrho}|\nabla^t f|_t^2
      +P_{s,t}^{\varrho}f^2\log P_{s,t}^{\varrho}f^2\notag\\
    &\quad+\int_s^t\l[2  \l(\int_r^t \kappa(u)\exp\l(-\int_r^uK(v)\,dv\r) du\r)^2+\sup \varrho_r^{+}\r]dr\,P_{s,t}^{\varrho}f^2.\label{eq-log-sobolev}
  \end{align}
\end{proof}

\begin{proof}[Proof of Theorem \ref{log-S-superbound}]
  % From the log-Sobolev inequality \eqref{eq-log-sobolev}
  % with respect to
  % $P_{s,t}^{\varrho}$ is
  % \begin{align}\label{eq-log-sobolev}
  %   P_{s,t}^{\varrho}(f^2\log f^2)
  %   &\leq 4\l(\int_s^t\exp\l(-2\int_r^tK(u)\, d u\r)\,d r\r)P_{s,t}^{\varrho}|\nabla^t f|_t^2+P_{s,t}^{\varrho}f^2\log P_{s,t}^{\varrho}f^2\notag\\
  %   &\quad+\int_s^t\l(2  \l(\int_r^t \kappa(u)\e^{-\int_r^uK(v)\,dv}\, du\r)^2+\sup \varrho_r^{+}\r)\,dr\cdot P_{s,t}^{\varrho}f^2.
  % \end{align}
  By means of the fact that
$$\log^{+}(P^{\varrho}_{s,t}f^2)\leq P^{\varrho}_{s,t}f^2\leq
\exp\l(\int_s^t\sup \varrho^-_u \,du\r)\|f\|_{\infty}^2,$$ we can
integrate both sides of the log-Sobolev inequality
\eqref{eq-log-sobolev} with respect to $\mu_s$.  Taking \eqref{Eq:mut}
into account, we get
\begin{align}\label{mu-f-1}
  \mu_t(f^2\log f^2)&\leq 4\left(\int_s^t\exp\l(-2\int_r^tK(u)\, d u\r) d r\right) \mu_t(|\nabla^tf|^2_t)+\mu_s(P^{\varrho}_{s,t}f^2\log P^{\varrho}_{s,t}f^2)\notag\\
                    &\quad+\int_s^t\l[2  \l(\int_r^t \kappa(u)\exp\l(-\int_r^uK(v)\,dv\r) du\r)^2+\sup \varrho_r^{+}\r]dr\,\mu_t(f^2).
\end{align}
We deal first with the term
$\mu_s(P^{\varrho}_{s,t}f^2\log P_{s,t}^{\varrho}f^2)$.  Let
$1<p<q$. For any $h\in {]0,1-1/p[}$ let
$$r_h=\frac{ph}{p-1}\in {]0,1[}.$$ By the Riesz-Thorin interpolation
theorem, we have
\begin{align}\label{add-eq-2}
  \big\|P_{s,t}^{\varrho}f\big\|_{q_h,s}\leq
  \big\|P_{s,t}^{\varrho}\big\|_{(p,t)\rightarrow (q,s)}^{r_h}\,\big\|P_{s,t}^{\varrho}\big\|_{(1,t)\rightarrow (1,s)}^{1-r_h}\,\big\|f\big\|_{p_h,t},\quad
  f\in L^p(M,\mu_s),
\end{align}
where
$$\frac1{p_h}=\frac{1-r_h}{1}+\frac{r_h}{p}\quad\text{and}\quad
\frac1{q_h}=\frac{1-r_h}{1}+\frac{r_h}q.$$ Thus
 $$p_h=\frac1{1-h}\,\quad \mbox{and}\,\quad  q_h=\l(1-\frac{p(q-1)}{q(p-1)}h\r)^{-1}.$$
 Since $\|P_{s,t}^{\varrho}\|_{(1,t)\rightarrow (1,s)}\leq 1$, we get
 from Eq.~\eqref{add-eq-2} that
 $$\int \big(P_{s,t}^{\varrho}|f|^{2(1-h)}\big)^{q_h}\,d \mu_s
 \leq \|P_{s,t}^{\varrho}\|_{(p,t)\rightarrow
   (q,s)}^{r_hq_h}\,\|f\|_{2,t}^{q_h/p_h}.$$ Then, for
 $f\in C_0^{\infty}(M)$ satisfying $\|f\|_{2,t}=1$, we have
 \begin{align}\label{Est:h}
   \frac1h&\l(\int \big(P^{\varrho}_{s,t}|f|^{2(1-h)}\big)^{q_h}\,d \mu_s-\l(\int P^{\varrho}_{s,t}|f|^2\, d \mu_s\r)^{q_h/p_h}\r)\notag\\
          &=\frac1{h}\l(\int \big(P^{\varrho}_{s,t}|f|^{2(1-h)}\big)^{q_h}\,d \mu_s-1\r)\\
          &\leq\frac1{h}\big(\|P_{s,t}^{\varrho}\|_{(p,t)\rightarrow (q,s)}^{r_hq_h}-1\big).\notag
 \end{align}
 Taking the limit as $h\to0$ in \eqref{Est:h}, as
 $$\lim_{h\rightarrow
   0}\frac1{h}\l(\|P_{s,t}^{\varrho}\|_{(p,t)\rightarrow
   (q,s)}^{r_hq_h}-1\r)=\frac{p}{p-1}\log
 \|P_{s,t}^{\varrho}\|_{(p,t)\rightarrow (q,s)},$$ we obtain by
 dominated convergence,
 \begin{align*}
   \frac{p(q-1)}{q(p-1)}\int P^{\varrho}_{s,t}f^2\log P^{\varrho}_{s,t}f^2\,d\mu_s
   -\int P_{s,t}^{\varrho}(f^2\log f^2)\, d \mu_s
   \leq \frac{p}{p-1}\log \|P_{s,t}^{\varrho}\|_{(p,t)\rightarrow (q,s)},
 \end{align*}
 or equivalently,
 \begin{align}\label{Est:hto0}\mu_s(P_{s,t}^{\varrho}f^2\log
   P_{s,t}^{\varrho}f^2)\leq \frac{q(p-1)}{p(q-1)}\,\mu_t(f^2\log
   f^2)+\frac{q}{q-1}\log \|P_{s,t}^{\varrho}\|_{(p,t)\rightarrow
     (q,s)}.
 \end{align}
 Substituting \eqref{Est:hto0} back into Eq.~\eqref{mu-f-1}, we arrive
 at
 \begin{align}\label{eq-9}
   \mu_t(f^2\log f^2)\leq \gamma(s,t) \,\mu_t(|\nabla^tf|_t^2)+\tilde{\beta}(s,t)
 \end{align}
 for every $f\in C_0^{\infty}(M)$ satisfying $\|f\|_{2,t}=1$, where
 \begin{align*}
   \gamma(s,t)
   &=\frac{4p(q-1)}{q-p}\int_{s}^t\exp\left(-2\int_r^tK(u)\, d u\right)\, d r,\notag\\
   \tilde{\beta}(s,t)
   &=\frac{pq}{q-p}\log\l( \|P_{s,t}^{\varrho}\|_{(p,t)\rightarrow (q,s)}\r)\\
   &\quad\  +\frac{p(q-1)}{q-p}\int_s^t\l(2 \l(\int_r^t \kappa(u)\exp\l(-\int_r^uK(v)\,dv\r) du\r)^2+\sup \varrho_r^{+}\r)\,dr.
   % \label{eq-5}
 \end{align*}
 % Note that $\tilde{\beta}(\newdot,t)$ is a positive function on
 % $[0,t]$ and 2\leqp\leq q$. ???
 We complete the proof by letting
 \begin{equation*}\beta_t(r)=\tilde{\beta}(\gamma_t^{-1}(r),t).\qedhere\end{equation*}
\end{proof}

\begin{theorem}\label{Sobolevineq-to-superbound}
  We keep the situation of Theorem
  $\ref{log-S-superbound}$, that is $\varrho_t$ bounded and
  \begin{align*}
    &\Ric_t+\Hess_t(\phi_t)-\frac{1}{2}\partial_tg_t\geq K(t),
      \quad |d \varrho_t|\leq\kappa(t) \quad \mbox{for }\  t\in [0,T[.
  \end{align*}
  Suppose there exists a function $\beta_t\colon \R^+\rightarrow
  \R^+$ such that
  \begin{align}\label{log-S-I-2}
    \mu_t(f^2\log f^2)\leq r\mu_t(|\nabla^tf|_t^2)+\beta_t(r),\quad \text{for all }\ t\in [0,T[\ \text{and }\ \|f\|_{2,t}=1.
  \end{align}
  Then $$\|P_{s,t}^{\varrho}\|_{(p,t)\rightarrow
    (q,s)}<\infty\quad\text{for $1<p<q\ $ and $\ 0\leq s\leq t<T$}.$$
\end{theorem}

\begin{proof}
  Let $0\leq s<t< T$ and $f\in C_0^{\infty}(M)$ such that
  $f\geq\delta>0$.  To calculate the derivative of
  $ \mu_s(P_{s,t}^{\varrho}f)^{q(s)}$ with respect to $s$, we start
  with some preparatory calculations:
  % By Lemma \ref{lem1}, we need to check the following to
  % handle the derivative of $ \mu_s(P_{s,t}^{\varrho}f)^{q(s)}$ with
  % respect to $s$:
  \begin{align}\label{Est:Drift}
    (L_s&+\partial_s)(P_{s,t}^{\varrho}f)^{q(s)}\notag\\
        &=L_s(P_{s,t}^{\varrho}f)^{q(s)}+q(s)\,(P_{s,t}^{\varrho}f)^{q(s)-1}(\partial_sP_{s,t}^{\varrho}f)+q'(s)(P_{s,t}^{\varrho}f)^{q(s)}\log P_{s,t}^{\varrho} f\notag\\
        &=q(s)(q(s)-1)(P_{s,t}^{\varrho}f)^{q(s)-2}|\nabla^sP_{s,t}^{\varrho}f|_s^2\notag\\
        &\quad+q'(s)(P_{s,t}^{\varrho}f)^{q(s)}\log P_{s,t}^{\varrho}f +q(s)\varrho_s(P_{s,t}^{\varrho}f)^{q(s)}.
  \end{align}
  By Corollary \ref{est-gradient-ineq}, there exist positive constants
  $c_1(s,t)$ and $c_2(s,t)$ such that
  $$\big\||\nabla^sP_{s,t}^{\varrho}f|_s^2\big\|_{\infty}^{\mathstrut}\leq c_1(s,t)\,\|f\|^2_{\infty}
  +c_2(s,t)\,\||\nabla^tf|_t\|_{\infty}^2.$$ Moreover,
  $\|P_{s,t}^{\varrho}f\|_{\infty}^{\mathstrut}\leq
  (P_{s,t}^{\varrho}1)\,\|f\|_{\infty}^{\mathstrut}$ and
  $$(P_{s,t}^{\varrho}f)^{q(s)}\log^{+}(P_{s,t}^{\varrho}f)
  \leq (P^{\varrho}_{s,t}f)^{q(s)+1}\leq
  (P_{s,t}^{\varrho}1)^{q(s)+1}\,\|f\|_{\infty}^{q(s)+1}.$$ Combining
  these estimates, we obtain
$$\big\|(L_s+\partial_s)(P_{s,t}^{\varrho}f)^{q(s)}\big\|_{\infty}^{\mathstrut}<\infty.$$
Now, by Theorem \ref{lem1}, we see that
\begin{align*}
  \frac{d}{d s}
  \mu_s\big((P^{\varrho}_{s,t}f)^{q(s)}\big)
  &=-\mu_s\big(\varrho_s(P_{s,t}^{\varrho}f)^{q(s)}\big)+\mu_s\big(\partial_s(P_{s,t}^{\varrho}f)^{q(s)}\big)\\  &=-\mu_s\big(\varrho_s(P_{s,t}^{\varrho}f)^{q(s)}\big)+\mu_s\big((L_s+\partial_s) (P_{s,t}^{\varrho}f)^{q(s)}\big)\\
  % &=\mu_s\big(-q(s)(P^{\varrho}_{s,t}f)^{q(s)-1}(L_sP^{\varrho}_{s,t}f)+q'(s)(P^{\varrho}_{s,t}f)^{q(s)}\log
  % P_{s,t}
  % f\big)-\big(1-q(s)\big)\,\mu_s\big(\varrho_s(P_{s,t}^{\varrho}f)^{q(s)}\big)\\
  &=q(s)(q(s)-1)\,\mu_s\big(|\nabla^sP^{\varrho}_{s,t}f|_s^2(P^{\varrho}_{s,t}f)^{q(s)-2}\big)+q'(s)\,\mu_s\big((P^{\varrho}_{s,t}f)^{q(s)}\log P_{s,t}^{\varrho}f\big)\\
  &\quad-\big(1-q(s)\big)\,\mu_s\big(\varrho_s(P_{s,t}^{\varrho}f)^{q(s)}\big).
\end{align*}
For $\|P_{s,t}f\|_{q(s),s}$, since
$\| P_{s,t}^{\varrho}f\|_{q(s),s}^{1-q(s)}=
\left(\mu_s((P_{s,t}^{\varrho}f)^{q(s)}\big)\right)^{1/q(s)-1}$, we
thus find
\begin{align}\label{add-eq-4}
  \frac{d}{d s}\| P_{s,t}^{\varrho}f\|_{q(s),s}
  &= (q(s)-1)\|P_{s,t}^{\varrho}f\|_{q(s),s}^{1-q(s)}\mu_s\big(|\nabla^sP_{s,t}^{\varrho}f|_s^2(P_{s,t}^{\varrho}f)^{q(s)-2}\big)\nonumber\\
  &\quad+\frac{q'(s)}{q(s)}\|P_{s,t}^{\varrho}f\|_{q(s),s}^{1-q(s)}\mu_s\big((P_{s,t}^{\varrho}f)^{q(s)}\log P_{s,t}^{\varrho}f\big)\nonumber\\
  &\quad-\frac{q'(s)}{q(s)}\|P_{s,t}^{\varrho}f\|_{q(s),s}\log \|P_{s,t}^{\varrho}f\|_{q(s),s}\notag\\
  &\quad+\frac{q(s)-1}{q(s)}\|P_{s,t}^{\varrho}f\|_{q(s),s}^{1-q(s)}\mu_s\big(\varrho_s(P_{s,t}^{\varrho}f)^{q(s)}\big)\notag\\
  &\geq (q(s)-1)\|P_{s,t}^{\varrho}f\|_{q(s),s}^{1-q(s)}\mu_s\big(|\nabla^sP_{s,t}^{\varrho}f|_s^2(P_{s,t}^{\varrho}f)^{q(s)-2}\big)\nonumber\\
  &\quad+\frac{q'(s)}{q(s)}\|P_{s,t}^{\varrho}f\|_{q(s),s}^{1-q(s)}\mu_s\big((P_{s,t}^{\varrho}f)^{q(s)}\log P_{s,t}^{\varrho}f\big)\nonumber\\
  &\quad-\frac{q'(s)}{q(s)}\|P_{s,t}^{\varrho}f\|_{q(s),s}^{1-q(s)}\,\|P_{s,t}^{\varrho}f\|_{q(s),s}^{q(s)}\log \|P_{s,t}^{\varrho}f\|_{q(s),s}\notag\\
  &\quad-\frac{(q(s)-1)\sup\varrho_s^-}{q(s)}\|P_{s,t}^{\varrho}f\|_{q(s),s}.
\end{align}
On the other hand, passing from $f$ to $f^{p/2}/\|f^{p/{2}}\|_{2,s}$
in the log-Sobolev inequality \eqref{log-S-I-2}, we obtain
$$\int f^p\log \l(\frac{f^p}{\|f^{p/{2}}\|_{2,s}^2}\r)d \mu_s\leq r \frac{p^2}{4}\int f^{p-2}|\nabla^sf|_s^2\, d \mu_s+\beta_s(r) \|f^{p/{2}}\|_{2,s}^2.$$
In this inequality, replacing $f$ and $p$ by $P_{s,t}^{\varrho}f$ and
$q(s)$ respectively, we obtain
\begin{align}\label{Eq:logSob}
  &\mu_s\l((P_{s,t}^{\varrho}f)^{q(s)}
    \log (P_{s,t}^{\varrho}f)\r)-\|P_{s,t}^{\varrho}f\|_{q(s),s}^{q(s)}\log \|P_{s,t}^{\varrho}f\|_{q(s),s}\notag\\
  &\quad\leq r \frac{q(s)}{4}\int (P_{s,t}^{\varrho}f)^{q(s)-2}|\nabla^s P_{s,t}^{\varrho}f|_s^2\, d \mu_s+\frac{\beta_s(r)}{q(s)}
    \|P_{s,t}^{\varrho}f\|_{q(s),s}^{q(s)}.
\end{align}
We now let
$$q(s)={\e}^{4r^{-1}(t-s)}(p-1)+1,\quad q(t)=p.$$
Note that $q$ is a decreasing function and $q'(s)r/4+(q(s)-1)\equiv0$.
Thus,  combining Eq.~\eqref{Eq:logSob} with Eq.~\eqref{add-eq-4}, we
arrive at
$$\frac{d}{d s}\|P^{\varrho}_{s,t}f\|_{q(s),s}\geq \l(\frac{\beta_s(r)q'(s)}{q(s)^2}-\frac{(q(s)-1)\sup\varrho_s^-}{q(s)}\r)\|P^{\varrho}_{s,t}f\|_{q(s),s},\quad 0\leq s\leq t<T.$$
It follows that
\begin{align}\label{Pf-eq-2}
  \|P_{s,t}^{\varrho}f\|_{q(s), s}\leq & \exp{\l(-\int_s^t\frac{\beta_u(r)q'(u)}{q(u)^2}\, d u+\int_s^t\frac{(q(u)-1)\sup\varrho_u^-}{q(u)}\, d u\r)}\,\|f\|_{p,t}.
\end{align}
If we impose that $q(s)=q$, then
$r=4(t-s)\l(\log \frac{q-1}{p-1}\r)^{-1}$.  Substituting the value of
$r$ into Eq.~\eqref{Pf-eq-2} yields
\begin{align*}
  \|P_{s,t}^{\varrho}f\|_{q,s}\leq
  \exp\Bigg(&-\int_s^t\frac{\beta_u(4(t-s) \l(\log (q-1)/(p-1)\r)^{-1})\,q'(u)}{q(u)^2}\, d u\\
            &+\int_s^t\frac{(q(u)-1)\sup\varrho_u^-}{q(u)}\, d u\Bigg)\,\|f\|_{p,t}.\qedhere
\end{align*}
\end{proof}

Using the dimension-free Harnack inequality, we are now in position to
give a necessary and sufficient condition for supercontractivity of
$P_{s,t}^{\varrho}$.

\begin{theorem}\label{th-2}
  We keep the situation of Theorem \textup{\ref{log-S-superbound}},
  that is $\varrho_t$ bounded and
  \begin{align*}
    &\Ric_t+\Hess_t(\phi_t)-\frac{1}{2}\partial_tg_t\geq K(t),
      \quad |d \varrho_t|\leq\kappa(t) \quad \mbox{for }\  t\in [0,T[.
  \end{align*}
  The condition
  $$\|P_{s,t}^{\varrho}\|_{(p,t)\rightarrow (q,s)}<\infty,
  \quad \text{for all }\ 1<p<q<\infty,\ 0\leq s\leq t< T,$$ then holds
  if and only if
  $$\mu_t\l(\exp(\lambda\rho_t^2)\r)<\infty\quad \text{for all \ $\lambda>0$ and
    $t\in [0,T[$}.$$
\end{theorem}
\begin{proof}
  By means of the Harnack inequality \eqref{Harnack-ineq1}, the
  theorem can be proved along the same lines as in \cite[Theorem
  5.7.3]{Wbook1} or \cite{ChT18}.  For the reader's convenience, we
  include a proof here.  First, for $0\leq s\leq t<T$, $p>1$ and
  $f\in C_b(M)$, it follows from the Harnack inequality
  \eqref{Harnack-ineq1} that

$$\l|(P_{s,t}^{\varrho}f)^p(x)\r|\leq (P_{s,t}^{\varrho}|f|^p)(y)\exp{\l((p-1)\int_s^t\sup\varrho_r^-\,dr +\frac{p\rho_s^2(x,y)}{4(p-1)\alpha(s,t)}+\frac{\eta(s,t)\rho_s(x,y)}{\alpha(s,t)}\r)}.$$
Thus, if $\mu_t(|f|^p)=1$, then
\begin{align}\label{est-P}
  1&\geq |P_{s,t}^{\varrho}f(x)|^p\int \exp{\l((1-p)\int_s^t\sup\varrho_r^-\,dr -\frac{p\rho_s^2(x,y)}{4(p-1)\alpha(s,t)}-\frac{\eta(s,t)\rho_s(x,y)}{\alpha(s,t)}\r)}\,\mu_s(d y)\nonumber\\
   &\geq  |P_{s,t}^{\varrho}f(x)|^p\,\mu_s(B_s(o,R))\notag\\
   &\quad\times\exp {\l((1-p)\int_s^t\sup\varrho_r^-\,dr -\frac{p(\rho_s(x)+R)^2}{4(p-1)\alpha(s,t)}
     -\frac{\eta(s,t)(\rho_s(x)+R)}{\alpha(s,t)}\r)},
\end{align}
where $B_s(o,R)=\{y\in M\colon\rho_s(y)\leq R\}$ denotes the geodesic
ball (with respect to the metric $g(s)$) of radius $R$ about $o\in M$
and where $\rho_t(\newdot)=\rho_t(o,\newdot)$.  Since
$\mu_t\big(\exp(\lambda\rho_t^2)\big)<\infty$, the system of measures
$(\mu_s)$ is compact, i.e., there exists $R=R(s)>0$, possibly
depending on $s$, such that
\begin{align*}
  \mu_s\big(B_s(o,R(s))\big)
  &=\mu_s\big(\{x: \rho_s(x)\leq R(s)\}\big)\geq 1-\frac{\mu_s(\rho_s^2)}{R(s)^2}
  % \geq 1-\frac{H_2(s)}{R(s)^2}
    \geq 2^{-p}
\end{align*}
(after normalising $\mu_s$ to a probability measure).  Combining the
last estimate with Eq.~\eqref{est-P}, we arrive at
$$
1\geq
|P_{s,t}^{\varrho}f(x)|^p\,2^{-p}\exp{\l((1-p)\int_s^t\sup\varrho_r^-\,dr-\frac{\eta(s,t)(\rho_s(x)+R)}{\alpha(s,t)}
  -\frac{p(\rho_s(x)+R)^2}{4(p-1)\alpha(s,t)}\r)}
$$
which further implies
\begin{align}\label{est-P-1}
  |P_{s,t}^{\varrho}f(x)|\leq 2\exp
  {\l(\frac{p-1}{p}\int_s^t\sup\varrho_r^-\,dr+\frac{\eta(s,t)(\rho_s(x)+R)}{p\alpha(s,t)}+\frac{(\rho_s(x)+R)^2}{4(p-1)\alpha(s,t)}\r)},\quad s<t.
\end{align}
Therefore, we achieve
$$\|P_{s,t}^{\varrho}f\|_{q,s}\leq \l(\mu_s\l(\exp\big(q(c_1+c_2\rho_s^2)\big)\r)\r)^{1/q}$$
for some positive constants $c_1, c_2$ depending on $s$ and $t$.
Hence, if $\mu_s(\exp(\lambda \rho_s^2))<\infty$ for any $\lambda>0$
and $s\in [0,T[$, then $P_{s,t}$ is supercontractive, i.e., for any
$1<p<q<\infty$,
$$\|P_{s,t}^{\varrho}\|_{(p,t)\rightarrow (q,s)}<\infty.$$

Conversely, if the semigroup $P_{s,t}^{\varrho}$ is supercontractive,
by Theorem \ref{log-S-superbound} the super log-Sobolev inequalities
\eqref{log-S-I-2} holds. We first prove that
$\mu_s(\e^{\lambda \rho_s})<\infty$ for $s\in [0,T[$ and $\lambda>0$.
To this end, let $\rho_{s,k}=\rho_s\wedge k$ and
$h_{s,k}(\lambda)=\mu_s(\exp{(\lambda \rho_{s,k})})$. Taking
$\exp(\lambda\rho_{s,k}/2)$ in the super log-Sobolev inequality
\eqref{log-S-I}, we obtain
\begin{align*}
  \lambda h'_{s,k}(\lambda)&-h_{s,k}(\lambda)\log h_{s,k}(\lambda)\leq h_{s,k}(\lambda)\lambda^2\l(\frac{r}{4}+\frac{\beta_s(r)}{\lambda^2}\r).
\end{align*}
This implies
\begin{align}\label{add-eq-5}
  \l(\frac1{\lambda}\log h_{s,k}(\lambda)\r)'=\frac{\lambda h'_{s,k}(\lambda)-h_{s,k}(\lambda)\log h_{s,k}(\lambda)}{\lambda^2 h_{s,k}(\lambda)}\leq \frac{r}{4}+\frac{\beta_s(r)}{\lambda^2}.
\end{align}
Integrating both sides of Eq.~\eqref{add-eq-5} from $\lambda$ to
$2\lambda$, we obtain
\begin{align}\label{eq-1}
  h_{s,k}(2\lambda)\leq h_{s,k}^2(\lambda)\exp\l(\frac{r}2\lambda^2+\beta_s(r)\r).
\end{align}
From this inequality, along with the fact that there exists a constant
$M_s$ such that
$$\mu_s(\{\lambda \rho_s\geq M_s\})\leq \frac1{4}\exp\l(-\frac{r}2\lambda^2-\beta_s(r)\r),$$
we get
\begin{align*}
  h_{s,k}(\lambda)
  &=\int_{\{\lambda \rho_s\geq M_s\}}\exp\l(\lambda \rho_{s,k}\r)\, d \mu_s
    +\int_{\{\lambda \rho_s<M_s\}}\exp\l(\lambda \rho_{s,k}\r) \, d \mu_s\\
  &\leq \mu_s(\{\lambda \rho_s\geq M_s\})^{1/2}\, \mu_s({\e}^{2\lambda \rho_{s,k}})^{1/2}+{\e}^{M_s}\mu_s(\{\lambda \rho_s<M_s\})\\
  &\leq \l(\frac1{4}\exp\l(-\frac{r}2\lambda^2-\beta_s(r)\r)\r)^{1/2}
    \exp\l(\frac{r}{4}\lambda^2+\frac12\beta_s(r)\r)h_{s,k}(\lambda)+{\e}^{M_s}\mu_s(\{\lambda \rho_s<M_s\})\\
  &\leq \frac12 h_{s,k}(\lambda)+{\e}^{M_s}\mu_s(\{\lambda \rho_s<M_s\}),
\end{align*}
which implies
$h_{s,k}(\lambda)\leq 2{\e}^{M_s}\mu_s(\{\lambda \rho_s<M_s\})$ for
$s\in [0,T[$. As $M_s$ is independent of $k$, letting $k$ tend to
infinity, we arrive at
\begin{equation*}
  \mu_s({\e}^{\lambda \rho_s})<\infty,\quad \text{for all }  s\in [0,T[.
\end{equation*}

To prove that moreover $\mu_s(\e^{\lambda \rho_s^2})<\infty$ for
$s\in [0,T[$ and $\lambda>0$, we can follow the argument in
\cite[pp. 22-23]{ChT18}.
\end{proof}

As an application, we apply the results above to the modified flow
$(M,g_t,\phi_t)$ evolving by
\begin{align*}
  \left\{
  \begin{aligned}
    &\partial_tg(x,\cdot)(t)=-2\big(h+\Hess(\phi)\big)(x,t); \\
    &\partial_t \phi_t=-\Delta_t\phi_t-\sH_t.
  \end{aligned}
      \right.
\end{align*}
For this system, we have $\varrho\equiv0$. Thus, by Theorem
\ref{log-S-superbound}, we get the following result.

\begin{corollary}\label{cor-log-Sobolev}
  Assume that $(g_t, \phi_t)$ follows the evolving equation
  \eqref{geoflow-2} and
  \begin{align*}
    &\Ric_t+\Hess_t(\phi_t)-\frac{1}{2}\partial_tg_t\geq K(t), \quad \mbox{for all}\ \ t\in [0,T[.
  \end{align*}
  Suppose that
  \begin{align*}
    \|P_{s,t}\|_{(p,t)\rightarrow (q,s)}<\infty \ \  \mbox{for} \ 1<p<q \ \mbox{and} \ 0\leq s\leq t<T.
  \end{align*}
  Then the super log-Sobolev inequalities
  \begin{align}\label{log-SI}
    \int f^2\log f^2\, d \mu_t\leq  r\, \int |\nabla^tf|_t^2\,d\mu_t+{\beta}_t(r),\quad r>0,
  \end{align}
  hold for every $f\in H^1(M,\mu_t)$ such that $\|f\|_{2,t}=1$,
  $t\in [0,T[$, where $\beta_t(r):=\tilde{\beta}(\gamma^{-1}_t(r),t)$ and
  \begin{align*}
    \tilde{\beta}(s,t)&=\frac{pq}{q-p}\log (\|P_{s,t}\|_{(p,t)\rightarrow (q,s)}),\\
    \gamma(s,t)
              &=\frac{4p(q-1)}{q-p}\int_{s}^t\exp\left(-2\int_r^tK(u)\,d u\right)\,d r,\notag\\
    \gamma^{-1}_t(r)&=\inf\{s\geq 0:\gamma(s,t)\leq r\}. \notag
  \end{align*}
\end{corollary}

The following result is a consequence of Theorem \ref{th-2}.

\begin{corollary}\label{cor-supercontractivity}
  Assume that $(g_t, \phi_t)$ evolves according to \eqref{geoflow-2}
  and that
  $$\Ric_t+\Hess_t(\phi_t)-\frac{1}{2}\partial_tg_t\geq K(t),\quad t\in [0,T[,$$ for some
  function $K\in C([0,T[)$. Then
  $$\|P_{s,t}\|_{(p,t)\rightarrow (q,s)}<\infty,\ \text{ for } 1<p<q<\infty\
  \text{ and }\ 0\leq s\leq t< T,$$ if and only if
  $$\mu_t\big(\exp(\lambda\rho_t^2)\big)<\infty,\ \text{ for }\lambda>0 \ \text{ and }\ t\in [0,T[.$$
\end{corollary}

% \proof[Acknowledgments]This work has been supported by the Fonds
% National de la Recherche Luxembourg (FNR) under the OPEN scheme
% (project GEOMREV O14/7628746).

\bibliographystyle{amsplain}%

\bibliography{references}

\end{document}